\documentclass[12pt]{amsart}
\usepackage{amsmath,latexsym,amsfonts,amssymb,amsthm}
\usepackage{geometry}
\usepackage{hyperref}
\usepackage{mathrsfs}
\usepackage{graphicx,color}
\usepackage[alphabetic]{amsrefs}
\usepackage[all,cmtip]{xy}
\usepackage{bbm, stmaryrd}
\usepackage{tabu}
\usepackage{enumitem}
\usepackage{tikz}
\usepackage{tikz-cd}

\newtheorem{prop}{Proposition}[section]
\newtheorem{thm}[prop]{Theorem}

\newtheorem{cor}[prop]{Corollary}
\newtheorem{conj}[prop]{Conjecture}
\newtheorem{lem}[prop]{Lemma}

\theoremstyle{definition}

\newtheorem{defn}[prop]{Definition}

\newtheorem{rem}[prop]{\it Remark}

\newtheorem*{claim*}{Claim}

\newcommand{\bC}{\mathbb{C}}
\newcommand{\bR}{\mathbb{R}}
\newcommand{\bA}{\mathbb{A}}
\newcommand{\bQ}{\mathbb{Q}}
\newcommand{\bZ}{\mathbb{Z}}
\newcommand{\bN}{\mathbb{N}}

\newcommand{\bG}{\mathbb{G}}
\newcommand{\bT}{\mathbb{T}}
\newcommand{\bk}{\mathbbm{k}}

\newcommand{\tX}{\widetilde{X}}
\newcommand{\tY}{\widetilde{Y}}
\newcommand{\tS}{\widetilde{S}}
\newcommand{\tD}{\widetilde{D}}
\newcommand{\tF}{\widetilde{F}}

\newcommand{\tE}{\widetilde{E}}

\newcommand{\cX}{\mathcal{X}}
\newcommand{\cY}{\mathcal{Y}}
\newcommand{\cZ}{\mathcal{Z}}

\newcommand{\cO}{\mathcal{O}}
\newcommand{\cL}{\mathcal{L}}
\newcommand{\cI}{\mathcal{I}}

\newcommand{\cF}{\mathcal{F}}
\newcommand{\cG}{\mathcal{G}}

\newcommand{\cE}{\mathcal{E}}
\newcommand{\cD}{\mathcal{D}}
\newcommand{\cB}{\mathcal{B}}

\newcommand{\cR}{\mathcal{R}}

\newcommand{\DMR}{\mathcal{DMR}}

\newcommand{\fa}{\mathfrak{a}}
\newcommand{\fb}{\mathfrak{b}}

\newcommand{\fm}{\mathfrak{m}}

\newcommand{\fab}{\fa_{\bullet}}

\newcommand{\Spec}{\mathrm{Spec}}
\newcommand{\Supp}{\mathrm{Supp}}

\newcommand{\mult}{\mathrm{mult}}
\newcommand{\lct}{\mathrm{lct}}

\newcommand{\vol}{\mathrm{vol}}
\newcommand{\ord}{\mathrm{ord}}
\newcommand{\Val}{\mathrm{Val}}

\newcommand{\Diff}{\mathrm{Diff}}

\newcommand{\gr}{\mathrm{gr}}

\newcommand{\QM}{\mathrm{QM}}
\newcommand{\Ex}{\mathrm{Ex}}
\newcommand{\LC}{\mathrm{LC}}
\newcommand{\hvol}{\widehat{\rm vol}}

\newcommand{\Proj}{\mathrm{Proj}}

\numberwithin{equation}{section}

\newcommand{\red}[1]{{\textcolor{red}{#1}}}

\title{Stable degenerations of singularities}

\date{}

\author{Chenyang Xu}
\address{Department of Mathematics, Princeton University, Princeton, NJ 08544, USA}
\email     {chenyang@princeton.edu}

\address   {Beijing International Center for Mathematical Research,       Beijing 100871, China}
\email     {cyxu@math.pku.edu.cn}

\author{Ziquan Zhuang}
\address{Department of Mathematics, Johns Hopkins University, Baltimore, MD 21218, USA}
\email{zzhuang@jhu.edu}

\begin{document}

\maketitle

\begin{abstract}
   For any Kawamata log terminal (klt) singularity and any minimizer of its normalized volume function, we prove that the associated graded ring is always finitely generated, as conjectured by Chi Li. As a consequence,  we complete the last step of establishing the Stable Degeneration Conjecture proposed by Chi Li and the first named author for an arbitrary klt singularity. 
\end{abstract}

\section{Introduction}

The algebraic K-stability of Fano varieties has been a very active research topic in recent years. The local analogue is a stability theory for Kawamata log terminal (klt) singularities.  There were two roots from earlier analytic works: the first is  the work \cite{MSY08}, which shows that for a klt singularity $x\in X$ with a torus $\bT$ action, the existence of the Sasaki-Einstein metric along a Reeb vector field $\xi$ implies that $\xi$ minimizes a normalized volume function defined among all Reeb vector fields; the second one is the study of the Ricci-flat metric tangent cone for a singularity on the K\"ahler-Einstein Fano variety which is the Gromov-Hausdorff limit of K\"ahler-Einstein manifolds \cite{DS17}.

The general algebraic local stability theory was initiated in \cite{Li18}. It centers around the problem of minimizing the \emph{normalized volume function} 
$$\hvol\colon \Val_{X,x}\to \bR_{>0}\cup\{+\infty\},$$ 
defined on the non-archimedean link $\Val_{X,x}$ (i.e. the set of valuations centered at $x$) of a klt singularity $x\in (X,\Delta)$.
The complete conjectural picture, named the Stable Degeneration Conjecture, was proposed in \cites{Li18, LX18} (see \cite[Conjecture 4.4]{Xu18}), which says for \emph{any} klt singularity, a minimizer of $\hvol$ induces a $\bT$-equivariant degeneration to a K-semistable Fano cone singularity (for some torus $\bT$); and moreover, such a degeneration is unique. 

\subsection{Main theorem}
After a period of intensive works, including \cites{Li18,Li17,LL19,Blum18,LX20,LX18,LWX-cone,Xu20,XZ-uniqueness} etc., the only unknown part of the Stable Degeneration Conjecture is the local higher rank finite generation, which we will address in the current paper, i.e. we confirm \cite[Conjecture 7.1(5)]{Li18}.

\begin{thm}[Local higher rank finite generation]\label{thm-localfinite}
Let $x\in (X=\Spec(R),\Delta)$ be a klt singularity. Let $v$ be a minimizer of $\hvol$ on $\Val_{X,x}$. Then the associated graded algebra ${\rm gr}_v R:=\bigoplus_{\lambda\in \bR_{\ge 0}}\fa_\lambda/\fa_{>\lambda}$ is finitely generated, where 
$$\fa_\lambda:=\{f\in R\ | \ v(f)\ge \lambda\}\mbox{\ \ \ \ \ and \ \ \ \ \ } \fa_{>\lambda}:=\{f\in R\ | \ v(f)>\lambda\}. $$
\end{thm}

Here when the value monoid $\Phi$ of the image $v\colon R\setminus\{ 0 \}\to \mathbb{R}_{\ge 0}$ generates an abelian group isomorphic to $\mathbb Z$, i.e. the rational rank of $v$ is one, the finite generation property follows from \cite{BCHM} (see \cites{Blum18, LX20}). However, the case when $v$ is of higher rational rank needs substantially more work.

As we mentioned above, together with the aforementioned earlier works, this completes the local stability theory of klt singularities. 
\begin{thm}[Stable Degeneration Conjecture]\label{thm-SDC}
Let $x\in (X={\rm Spec}(R),\Delta)$ be a klt singularity. 
\begin{enumerate}
    \item Up to rescaling, $\hvol$ has a unique minimizer $v$ in $\Val_{X,x}$ which is quasi-monomial with a finitely generated associated graded ring $\gr_v R$, and it induces a degeneration of $(X,\Delta)$ to a K-semistable log Fano cone singularity $(X_0={\rm Spec}({\rm gr}_v R),\Delta_0;\xi_v)$, where $\xi_v$ is the Reeb vector field induced by the grading of $\gr_v R$. 
    \item Conversely, assume a quasi-monomial valuation $v$ has a finite generated associated graded ring ${\rm gr}_v R$ such that $\left(X_0={\rm Spec}({\rm gr}_v R),\Delta_0\right)$ is klt where $\Delta_0$ is the degeneration of $\Delta$ and $(X_0,\Delta_0;\xi_v)$ is K-semistable as a log Fano cone singularity. Then $v$ is a minimizer of $\hvol$ on $\Val_{X,x}$.
    \item $(X_0,\Delta_0;\xi_v)$ has a degeneration to a unique K-polystable log Fano cone singularity $(Y,\Delta_Y;\xi)$, which admits a Ricci-flat K\"ahler cone metric when the ground field $\bk=\bC$.  
\end{enumerate}
\end{thm}

We refer to \cite[Section 2.2]{LWX-cone} for the definition of log Fano cone singularities and K-stability notions for them.

The above theorem combines many previous works together with Theorem \ref{thm-localfinite}. More precisely, the normalized volume function was first introduced in \cite{Li18}. The existence of a minimizer $v$ was proved in \cite{Blum18}, which was shown to be quasi-monomial in \cite{Xu20}. Then \cite{XZ-uniqueness} proved that a minimizer is unique up to rescaling. In \cites{Li17, LX20}, the rest of (1) and (2), i.e. the K-semistable characterization of the minimizer, was proved for the case when the rational rank of $v$ is equal to one.
Theorem \ref{thm-localfinite} addresses the finite generation of the associated graded ring for $v$ in a general case. Assuming this  \cite{LX18} proves the K-semistable characterization  in all cases. The existence and uniqueness of a K-polystable degeneration in (3) is proved in \cite{LWX-cone}. 
When $k=\mathbb{C}$,  such a degeneration admits a Ricci-flat K\"ahler cone metric by \cite[Theorem 2.9]{Li21}, which in turn relies on the global finite generation theorem proved in \cite{LXZ-HRFG}.

As a consequence we also get the following result:

\begin{thm}\label{t-algebraicvolume}
Let $x\in (X,\Delta)$ be a klt singularity. Then its local volume
$$\hvol(x,X,\Delta):= \inf_{v\in \Val_{X,x}} \hvol(v)$$
is an algebraic number. 
\end{thm}

\subsection{Finite generation via multiple degeneration}

Higher rank finite generation plays a key role to complete the algebraic theory of K-stability. In the global setting, one needs to show that for any log Fano pair $(X,\Delta)$ that is not uniformly K-stable, any valuation $v$ computing the $\delta$-invariant has a finitely generated associated graded ring $\gr_v(R)$ for $R=\bigoplus_{m\in r\cdot\bN}H^0(-m(K_X+\Delta))$ where $r>0$ is a sufficiently divisible integer. This was solved in \cite{LXZ-HRFG}, and the proof can be split into two parts. In the first part, it is shown that such a valuation $v$ is a monomial lc place on a log resolution $(Y,E)$ of $(X,\Delta)$, for a $\bQ$-complement $\Gamma$ whose birational transform on $Y$ contains an ample $\bQ$-subdivisor. Such a complement $\Gamma$ is called special (with respect to the log resolution $(Y,E)$). It is a key observation in \cite{LXZ-HRFG} that when the rational rank of $v$ is greater than one, one should focus on lc places of this smaller class of $\bQ$-complements to achieve the finite generation. Then in the second part, it is shown that any $v$ satisfying the above assumption indeed induces a finitely generated associated graded ring ${\rm gr}_vR$. One major input in the proof of the second part is to show that for any small rational perturbation $w$ of $v$, the induced degeneration $X_w:={\rm Proj} \Big( {\rm gr}_w(R)\Big)$ yields a log Fano pair with an $\alpha$-invariant bounded from below by a constant (which does not depend on $w$), therefore belongs to a bounded family of pairs by \cite{Jiang-boundedness}. 

For the local problem, one can mimic the argument in the first step as above to construct a special $\mathbb Q$-complement for a valuation minimizing $\hvol$, and then perturbations of $v$ yields degenerations to log Fano cones. However, a main issue is that the boundedness result for such log Fano cones, unlike the case for Fano varieties, is not known (see e.g. Conjecture \ref{conj-bounded}). Therefore, we have to develop a new approach to circumvent the usage of boundedness result. In \cite{Xu-HRFG}, it was proposed to use multiple degenerations over a base $\mathbb A^n$, to attack the finite generation. More precisely, if $\Gamma$ is a $\bQ$-complement of a klt singularity $x\in (X,\Delta)$, then \cite{Xu-HRFG} shows that one can use any $r$ divisorial lc places of $(X,\Delta+\Gamma)$ to obtain a $\mathbb{G}_m^r$-equivariant flat degeneration of $(X,\Delta)$ over $\mathbb{A}^r$, which indeed yields a locally stable family in the sense of \cite{Kol23}*{Definition-Theorem 4.7}. 

The key recipe for this strategy to work is to show that the degeneration family has irreducible fibers over all points in $\mathbb A^r$. This is in general a delicate condition. One of our main observations in this paper is that a refinement of the notion of the special $\mathbb Q$-complement can be used to guarantee that the degenerations induced by its lc places have irreducible fibers. 
More precisely, besides the log smooth model considered in the original definition of special $\bQ$-complements in \cite{LXZ-HRFG},  we also consider models 
$$\pi\colon (Y, E={\rm Ex}(\pi))\to X$$
such that $\pi$ is an isomorphism over $X\setminus \{x\}$, $(Y,E+\pi_*^{-1}\Delta)$ is qdlt (Definition \ref{defn:qdlt}), and $-(K_Y+E+\pi_*^{-1}\Delta)$ is ample.  
We call this \emph{a Koll\'ar model}, as it is a generalization of Koll\'ar component to the case of multiple exceptional components. 
If we write $E=\sum^r_{i=1}E_i$, then a key technical result says that the multiple degenerations over $\bA^r$ given by $E_i$ always have irreducible central fibers. 
This is sufficient to imply the finite generation of the associated graded ring for any valuation $v\in {\rm QM}(Y,E)$, as we will see in Section \ref{ss-irredu=finite}. 

Then starting from any log smooth model $(Y,E)$ admitting a special $\bQ$-complement $\Gamma$ and a monomial lc place 
$$v\in \QM(Y,E)\cap {\rm LC}(X,\Delta+\Gamma),$$ we will show that $v$ is contained in $\QM(Y',E')$ for a Koll\'ar model $(Y',E')\to X$. For this, we find a process to modify $(Y,E)$ by first blowing up, and then running a minimal model program. See Section \ref{ss-constructingmodel}.

To summarize, we show the following technical statement (Theorem \ref{t-localHRFG}).

\begin{thm}\label{thm-monomialfinitegeneration}
Let $x\in (X={\rm Spec}(R),\Delta)$ be a klt singularity and let $v$ be a monomial lc place of a special $\bQ$-complement $\Gamma$ with respect to some log smooth model $(Y,E)$ of $(X,\Delta)$ $($Definition \ref{d-specialcomplement}$)$. Then the associated graded ring ${\rm gr}_v(R)$ is finitely generated. 
\end{thm}

We remark that a special case of this statement, i.e. when $x\in (X,\Delta)$ is a log Fano cone singularity (with respect to some torus $\bT$) and $v$ is $\bT$-invariant, is proved in \cite{Huang-thesis} by improving arguments in \cite{LXZ-HRFG}.

Theorem \ref{thm-monomialfinitegeneration} implies the following global version of the higher rank finite generation statement, which was first proved in \cite[Theorem 4.2]{LXZ-HRFG}. 

\begin{cor}[{\cite[Theorem 4.2]{LXZ-HRFG}}]\label{cor-global}
Let $(X,\Delta)$ be a log Fano pair. Let $v$ be a monomial lc place of a special $\bQ$-complement $\Gamma$ with respect to some log smooth model $\pi\colon (Y,E)\to (X,\Delta)$. Then the associated graded algebra ${\rm gr}_v(R)$ is finitely generated, where $R=\bigoplus_{m\in \bN} H^0(-mr(K_X+\Delta))$ for some $r$ such that $r(K_X+\Delta)$ is Cartier.
\end{cor}

Along the way of the proof, we also get the following more precise characterization for how the induced degenerations vary as we vary the valuations.

\begin{thm}\label{thm-finitemodel}
Let $x\in (X={\rm Spec}(R),\Delta)$ be a klt singularity and let $(Y,E)$ be a toroidal model over $(X,\Delta)$ with a special $\bQ$-complement $\Gamma$. Then $\LC_x(\Gamma;Y,E)$ is a rational polyhedral subset of $\QM(Y,E)$, and for any two valuations $v$ and $w$ on the same open face of $\LC_x(\Gamma;Y,E)$, there is an isomorphism $\gr_v R\cong \gr_w R$ of the associated graded algebras. 
\end{thm}

For the definition of toroidal models (resp. $\LC_x(\Gamma;Y,E)$), see Section \ref{sss:models} (resp. Definition \ref{d-specialcomplement}).

\medskip

The remaining part of the proof is showing that any minimizer $v$ is a monomial lc place for a log smooth model $(Y,E)$ over $(X,\Delta)$ with a $\bQ$-complement. This is similar to \cite[Corollary 3.4]{LXZ-HRFG}, but we replace the basis type divisors in the global settings by the construction in \cite{XZ-uniqueness}. See Section \ref{ss-specialcomplement}.  

Finally, we want to restate the following conjecture that is well known to experts. 
\begin{conj}\label{conj-bounded}
Fix a positive integer $n$, a positive number $\delta$ and a finite set $I\subseteq [0,1]\cap \bQ$. Then the set
\[
\Big\{
(X,\Delta)\  \Big|\begin{array}{l} x\in (X,\Delta) \mbox{ is a K-semistable log Fano cone singularity},\\ 
\dim X=n,\, {\rm Coeff}(\Delta)\subseteq I,\, \mbox{and } \hvol(x,X,\Delta)\ge \delta\end{array}
\Big\}
\]
is bounded.
\end{conj}
Theorem \ref{thm-SDC} and Conjecture \ref{conj-bounded} together imply that for any finite set $I$ of rational numbers and any positive number $\delta$, the class of all $n$-dimensional klt singularities $x\in (X,\Delta)$ with $\hvol(x,X,\Delta)\ge \delta$ and ${\rm Coeff}(\Delta)\subseteq I$ is bounded up to special degeneration (i.e. \emph{specially bounded}). They also imply that the local volumes of $n$-dimensional klt singularities with ${\rm Coeff}(\Delta)\subseteq I$ satisfies the ascending chain condition (ACC) and the only accumulation point is zero. This is because Theorem \ref{thm-SDC} yields a degeneration of $(x,X,\Delta)$ to a K-semistable Fano cone singularity $(X_0,\Delta_0;\xi_v)$ and we know that 
$$\hvol(x,X,\Delta)=\hvol_{(X,\Delta)}(v)=\hvol_{(X_0,\Delta_0)}(\xi_v)=\hvol(X_0,\Delta_0;\xi_v),$$
where the second equality follows from \cite[Lemma 2.58]{LX18}. As we mentioned, a major difference between our current approach with the one in \cite{LXZ-HRFG} is that we circumvent the above type of boundedness result, by considering multiple degenerations. Nevertheless, we believe Conjecture \ref{conj-bounded} still plays a key role for advancing our understanding of the local K-stability theory.
\medskip

\noindent {\bf Acknowledgement}: We would like to thank Chi Li, Yuchen Liu for helpful conversations, and the referees for careful reading and helpful comments. CX is partially supported by NSF Grant DMS-2201349, DMS-2139613 and DMS-2153115. ZZ is partially supported by NSF Grants DMS-2240926, DMS-2234736, a Clay research fellowship, as well as a Sloan fellowship. 

\section{Preliminaries}\label{s-prelim}

Throughout this paper, we work over an algebraically closed field $\bk$ of characteristic $0$. All algebraic varieties are assumed to be quasi-projective. We follow the standard terminology from \cites{KM98, Kol13}.


A singularity $x\in (X,\Delta)$ consists of a normal affine variety $X$, an effective $\bQ$-divisor $\Delta$ (could be $0$) on $X$ and a closed point $x\in X$. A divisor over a singularity $x\in (X,\Delta)$ is a prime divisor $E\subseteq Y$ on some proper birational model $\mu\colon Y\to X$ such that $\mu(E)=\{x\}$.

For an effective $\bQ$-divisor $\Delta$ on a normal variety,  an effective $\bQ$-divisor $\Delta_0$ is called \emph{a $\bQ$-subdivisor of $\Delta$} if $\Delta_0\le \Delta$.

Let $(X,\Delta)\to U$ be a projective morphism from a klt pair $(X,\Delta)$ such that $-K_X-\Delta$ is ample. A (\emph{global}) \emph{$\mathbb{Q}$-complement} is an effective $\mathbb{Q}$-divisor $\Gamma$, such that $(X,\Delta+\Gamma)$ is log canonical and $K_X+\Delta+\Gamma\sim_{\bQ} 0$.
A (\emph{local}) \emph{$\bQ$-complement} of a singularity $x\in (X,\Delta)$ is an effective $\bQ$-divisor $\Gamma$ such that $(X,\Delta+\Gamma)$ is lc and $x$ is an lc center.\footnote{Note that our definition of complements is slightly different from Shokurov's original definition \cite{Shokurov92}.}

Given a ring $R$, an ($\bN$-)\emph{graded sequence $\fa_{\bullet}=\{\fa_k\}_{k\in \bN}$}  of ideals is a set of ideals $\fa_{k}\subseteq R$ $(k\in \bN)$ satisfying that $\fa_{k}\cdot \fa_{k'}\subseteq \fa_{k+k'}$. 

\subsection{Qdlt pairs} \label{ss:qdlt}

The concept of qdlt pairs is introduced in \cite[Definition 35]{dFKX-dualcomplex}.

\begin{defn} \label{defn:toroidal}
Let $D\subseteq X$ be a reduced divisor on a normal variety. We say $(X,D)$ is \emph{toroidal} if for any $x\in X$ there exists an SNC pair $x'\in (X',D')$ and a finite abelian group $G$ acting on it preserving every irreducible component of $D'$, such that the $G$-action is free in codimension one and Zariski locally we have
\[
\left( x\in (X,D)\right)\cong \left( x'\in (X',D')\right)/G.
\]
\end{defn}

\begin{defn}[\cite{dFKX-dualcomplex}*{Definition 35}] \label{defn:qdlt}
A \emph{qdlt} pair $(X,\Delta)$ is a log pair such that there exists an open set $U\subset X$ such that $\Delta|_U$ is reduced, $(U,\Delta|_U)$ is toroidal, and the log discrepancy $A_{X,\Delta}(F)>0$ for any prime divisor $F$ over $X$ whose center on $X$ is contained in $X\setminus U$. Note that when $\lfloor\Delta\rfloor$ is irreducible, being qdlt is equivalent to being plt. Similarly we can define \emph{sub-qdlt} pairs by allowing negative coefficients in $\Delta$. 
\end{defn}

Qdlt pairs behave very similar to dlt pairs. We list a few properties of this kind.

\begin{lem} \label{lem:perturb qdlt boundary} 
Assume that $(X,\Delta)$ is a qdlt pair and $X$ is quasi-projective. Let $\Delta_0$ and $\Delta_1$ be effective $\bQ$-divisors $($not necessarily $\bQ$-Cartier$)$. Assume that $\Delta_0\le \Delta$ and $\Supp(\Delta_1)$ does not contain any stratum of $\lfloor \Delta \rfloor$. Then there exist some $0<\varepsilon\ll 1$ and some $\bQ$-divisor $D\ge \Delta_0+\varepsilon \Delta_1$ such that $(X,D)$ is qdlt and $\lfloor D\rfloor = \lfloor \Delta_0\rfloor$. Furthermore, if $X$ is projective over $S$ and $-(K_X+\Delta)$ is ample over $S$, then $D$ can be chosen so that $-(K_X+D)$ remains ample over $S$.
\end{lem}

In particular, we can always make small perturbation of a qdlt pair to reduce the number of components with coefficient $1$.

\begin{proof}
This is proved just as in \cite{KM98}*{Proposition 2.43}. Let $H$ be an ample divisor on $X$, and let $U$ be the open set as in Definition \ref{defn:qdlt}. Since $\Supp(\Delta_1)$ does not contain any stratum of $\lfloor \Delta \rfloor$, after possibly shrinking $U$ we may assume that $\Delta_1|_U=0$. Choose some $r>0$ such that $r(\Delta-\Delta_0-\Delta_1)$ has integer coefficients and $r(\Delta-\Delta_0-\Delta_1)|_U$ is Cartier. By the following Lemma \ref{lem-noncartierlocus}, we may choose some $m\gg 0$ such that the linear system $|mH+r(\Delta-\Delta_0-\Delta_1)|$ is basepoint free on $U$. Let $G_0\in |mH+r(\Delta-\Delta_0-\Delta_1)|$ be a general member and let $G=\frac{1}{r}G_0$. Then 
\[
K_X+\Delta_0+\Delta_1+G\sim_\bQ K_X+\Delta+\frac{m}{r}H
\]
is $\bQ$-Cartier, and by Bertini's theorem $(X,\Delta+\Supp(G))$ is toroidal on $U$. Consider the convex combination
\[
(X,D_\varepsilon:=\Delta_0+(1-\varepsilon)(\Delta-\Delta_0)+\varepsilon (\Delta_1+G))
\]
of $(X,\Delta)$ and $(X,\Delta_0+\Delta_1+G)$. This pair has toroidal support on $U$ and $\lfloor D_\varepsilon\rfloor=\lfloor \Delta_0\rfloor$ when $0<\varepsilon\ll 1$. If $\pi\colon Y\to (X,\Delta+\Delta_1+G)$ is a log resolution and $F\subseteq Y\setminus\pi^{-1}(U)$ is a prime divisor, then $A_{X,D_\varepsilon}(F)\to A_{X,\Delta}(F)>0$ as $\varepsilon\to 0$. Thus $(X,D_\varepsilon)$ is qdlt when $0<\varepsilon\ll 1$ and $D_\varepsilon$ gives the desired boundary. Moreover, 
\[
K_X+D_\varepsilon\sim_\bQ K_X+\Delta+\frac{m\varepsilon}{r}H,
\]
hence $-(K_X+D_\varepsilon)$ remains ample if $-(K_X+\Delta)$ is ample. 
\end{proof}

\begin{lem}\label{lem-noncartierlocus}
Let $X$ be a normal variety, and $D$ a Weil divisor on $X$. Let $H$ be an ample divisor. Then the base locus of $|mH+D|$ for $m\gg 0$ is the non-Cartier locus of $D$. 
\end{lem}

\begin{proof}
This should be well-known to experts. First we show that every closed point $x\in X$ in the non-Cartier locus of $D$ is also in the base locus of $|mH+D|$. In fact, if there exists $G\in |mH+D|$ such that $x\not\in \Supp(G)$ then $G$ is clearly Cartier at $x$, hence so is $D\sim G-mH$, a contradiction. Next we take some $m\gg 0$ such that the sheaf $\cO_X(D+mH)$ is globally generated. For every $x\in X$ where $D$ is Cartier, we have $\cO_X(D+mH)\otimes k(x)\cong k(x)$, while global generation induces a surjection
\[
H^0(X,\cO_X(D+mH))\twoheadrightarrow 
\cO_X(D+mH)\otimes k(x).
\]
Thus we may find a section $s\in H^0(X,\cO_X(D+mH))$ that does not vanish at $x$; in other words, we get a divisor $G\in |mH+D|$ that does not pass through $x$. As $x$ is arbitrary, this implies that the base locus of $|mH+D|$ is contained in the non-Cartier locus of $D$, and we are done.
\end{proof}

The next lemma describes the geometry of the lc centers of qdlt pairs.

\begin{lem} \label{lem:qdlt->normal lc center}
Let $(X,\Delta)$ be a qdlt pair. Then the lc centers of $(X,\Delta)$ are normal, and they are exactly the irreducible components of various intersections $D_1\cap \dots\cap D_r$ where $D_1,\dots,D_r$ are irreducible components of $\lfloor\Delta\rfloor$.
\end{lem}

\begin{proof}
By definition the lc centers intersect the toroidal locus $U$ and then it is not hard to see that any lc center is given in this form. Let us show that any such intersection $D_1\cap \dots\cap D_r$ is normal.  Using induction on $r$ it suffices to show that each component $D_i$ is normal, since by adjunction, $(D_1,\Diff_{D_1}(\Delta-D_1))$ is qdlt. Without loss of generality we assume $i=1$. By Lemma \ref{lem:perturb qdlt boundary}, there exists some $\Delta'\ge D_1$ such that $(X,\Delta')$ is qdlt and $\lfloor\Delta'\rfloor=D_1$ (which is irreducible). By definition this implies that $(X,\Delta')$ is plt, thus by \cite{Kol13}*{Theorem 4.16} we see that $D_1$ is normal and we are done.
\end{proof}

We will sometimes use the following lemma to verify a given pair is qdlt.

\begin{lem}[\cite{dFKX-dualcomplex}*{Proposition 34}] \label{lem:qdlt criterion}
Let $(X,\Delta)$ be an lc pair. Then it is qdlt if and only if for any minimal lc center $W$ of $(X,\Delta)$, there exist $r={\rm codim}_X(W)$ reduced divisors $D_1,\dots,D_r\subseteq \lfloor \Delta\rfloor$ that are $\bQ$-Cartier at the generic point of $W$ such that $W\subseteq D_i$.
\end{lem}

\subsection{Valuations and models}
\subsubsection{Valuations}
Let $X$ be a reduced, irreducible (separated) variety defined over $\bk$. A \emph{real valuation} of its function field
$K(X)$ is a non-constant map 
\[
v\colon K(X)^{\times}(:= K(X)\setminus\{0\})\to \bR
\]
satisfying:
\begin{itemize}
 \item $v(fg)=v(f)+v(g)$;
 \item $v(f+g)\geq \min\{v(f),v(g)\}$;
 \item $v(\bk^{\times})=0$.
\end{itemize}
We set $v(0)=+\infty$. The {\it rational rank} ${\rm rat.rk.}(v)$ of a valuation $v$ is the dimension of the $\bQ$-vector space in $\bR$ spanned by $v(K(X)^\times)$. A valuation $v$ gives rise
to a valuation ring 
$$\cO_v:=\{f\in K(X)\mid v(f)\geq 0\}.$$
We say a real valuation $v$ is centered at a scheme-theoretic
point $x=c_X(v)\in X$ if we have a local inclusion 
$\cO_{X,x}\hookrightarrow\cO_v$ of local rings.
Notice that the center of a valuation, if exists,
is unique since $X$ is separated. Denote by $\Val_X$ 
the set of real valuations of $K(X)$ that admits a center
on $X$. For a closed point $x\in X$, we denote by $\Val_{X,x}$ the set
of real valuations of $K(X)$ centered at $x\in X$. Note that a valuation $v\in \Val_{X}$ is centered at a closed point $x\in X$ if and only if $v(f)>0$ for all $f\in \fm_x$. 

For each valuation $v\in \Val_{X,x}$ and any non-negative real number $\lambda$, we define the valuation ideal 
$$\fa_{\lambda}(v):=\{f\in\cO_{X,x}\mid v(f)\geq \lambda\}.$$ Then it is clear that $\fa_{\lambda}(v)$ is an $\fm_x$-primary ideal for each $\lambda>0$ and $x=c_{X}(v)$.

\medskip

Given a valuation $v\in \Val_X$ and a nonzero ideal $\fa\subseteq\cO_X$, we may evaluate $\fa$ along $v$ by setting 
$$v(\fa) := \min\{v(f)\mid f \in \fa\cdot\cO_{X,c_X(v)} \}.$$
It follows from the above definition that if $\fa\subseteq \fb \subseteq \cO_X$ are nonzero ideals, then $v(\fa) \geq v(\fb)$.
Additionally, $v(\fa)> 0$ if and only if $c_X(v) \in {\rm Supp}(\cO_X/\fa)$.
We endow $\Val_X$  with the weakest topology such that,
for every nonzero ideal $\fa$ on $X$, the map $\Val_X\to \bR$ defined by $v\mapsto v(\fa)$ is continuous.
The subset $\Val_{X,x}\subseteq \Val_X$ is endowed with
the subspace topology. In some literature, the space $\Val_{X,x}$ is called the {\it non-archimedean link} of $x\in X$. For more background on valuation spaces, see \cite[Section 4]{JM12}.

Let $D$ be a Cartier divisor on $X$ and $v\in \Val_X$ a valuation, we can also define $v(D)$ to be $v(f)$ where $f$ is a defining equation of $D$ at $c_X(v)$. And similarly if $D$ is $\bQ$-Cartier, we can define $v(D)=\frac{1}{m}v(mD)$ for a sufficiently divisible positive integer $m$. For a graded sequence $\fa_{\bullet}=\{\fa_k\}_{k\in \bN}$ of ideals, we define 
$$v(\fa_{\bullet})=\inf_{k}\frac{v(\fa_k)}{k}=\lim_{k\to \infty, \fa_k\neq 0}\frac{v(\fa_k)}{k}.$$

Let $\mu\colon Y\to X$ be a proper birational morphism with $Y$ a normal variety. For a prime divisor $E$ on $Y$ (we call any such $E$ a divisor over $X$), we define a valuation $\ord_E\in \Val_X$ that sends each rational function in $K(X)^{\times}=K(Y)^{\times}$ to its order of vanishing along $E$. Note that 
the center $c_X(\ord_E)$ is the generic point of $\mu(E)$.
We say that $v\in \Val_X$ is a \emph{divisorial valuation} if there exists $E$ as above and $\lambda\in\bR_{>0}$ such that $v=\lambda\cdot \ord_E$.

When $X=\Spec(R)$ is affine and $v\in \Val_X$ is a real valuation, we define the associated graded algebra as
\[
\gr_v R:=\bigoplus_{\lambda\in\bR} \fa_\lambda(v)/\fa_{>\lambda}(v)
\]
where $\fa_{>\lambda}(v)=\bigcup_{\mu>\lambda} \fa_\mu(v)=\{f\in R\ |\  v(f)>\lambda\}$.

\begin{lem} \label{lem:graded integral->valuation}
Let $R$ be an integral domain and let $\fab$ be a graded sequence of ideals such that $\fa_{i+1}\subseteq \fa_i$ for all $i\in \bN$. Suppose that $\gr(R,\fab):=\bigoplus_{i\in\bN} \fa_i/\fa_{i+1}$ is an integral domain. Then there exists some discrete valuation $v$ such that $\fa_i = \fa_i (v)$ for all $i\in\bN$.
\end{lem}

\begin{proof}
For any $0\neq f\in R$ set $v(f)=\sup\{i\,|\, f\in \fa_i\}$. We need to show that $v$ is a valuation. Since $\fa_{i+1}\subseteq \fa_i$ for all $i\in \bN$ ,we clearly have $v(f+g)\ge \min\{v(f),v(g)\}$. Suppose that $f\in \fa_i\setminus \fa_{i+1}$ and $g\in \fa_j\setminus \fa_{j+1}$, then their reductions $\bar{f}\in \fa_i/\fa_{i+1}$ and $\bar{g}\in \fa_j/\fa_{j+1}$ are non-zero in $\gr(R,\fab)$. Since $\gr(R,\fab)$ is an integral domain we get $\bar{f}\bar{g}\neq 0$, which gives $fg\in \fa_{i+j}\setminus \fa_{i+j+1}$ and hence $v(fg)=v(f)+v(g)$.
\end{proof}

\subsubsection{Models} \label{sss:models}

A model $(Y,E)$ over a pair $(X,\Delta)$ consists of a projective birational morphism $\pi\colon Y\to X$ and a reduced divisor $E$ on $Y$. We call it a \emph{log smooth} (resp. \emph{toroidal}) \emph{model} if $(Y,\Supp(E+\Ex(\pi)+\pi_*^{-1}\Delta))$ is simple normal crossing (SNC) (resp. toroidal).



\begin{defn}\label{d-quasimono}
A valuation $v=v_{\alpha}$ is called quasi-monomial if there exists a log smooth model $(Y,E=E_1+\dots+E_r)\to X$ and $r$ nonnegative numbers $\alpha=(\alpha_1$,\dots, $\alpha_r)$ such that 
\begin{enumerate}
\item $\bigcap^r_{i=1} E_i\neq \emptyset$;
\item There exists a component $C\subseteq \bigcap^r_{i=1} E_i$, such that around the generic point $\eta$ of $C$, $E_i$ is given by the local equation $(z_i=0)$ and 
\item  for any $f\in \mathcal{O}_{Y,\eta}$, we may write $f=\sum_{\beta \in \bZ^r_{\ge 0}} c_{\beta}z^{\beta}$ (where $z^\beta = z_1^{\beta_1}\cdots z_r^{\beta_r}$) in $\hat{\mathcal{O}}_{Y,\eta}$ so that each $c_{\beta}$ is either zero or a unit,
then 
$$v_{\alpha}(f)=\min \left\{\sum \alpha_i\beta_i\,|\, c_{\beta}\neq 0\right\}.$$
\end{enumerate}
\end{defn}

For example, if we choose $\alpha\in \mathbb{Z}^r_{\ge 0}$, the corresponding valuation is given by the toroidal divisor coming from the weighted blow up with weight $\alpha$.
The rational rank of the quasi-monomial valuation $v_\alpha$ is equal to the dimension of the 
$\mathbb{Q}$-linear vector space 
$${\rm span}_{\mathbb{Q}}\{\alpha_1,\dots, \alpha_r\}\subseteq \mathbb{R}.$$ 
In particular, the rational rank is always bounded above by the codimension of $\eta$. On the other hand, by further blowing up we can always choose a log smooth model so that $${\rm codim}_Y(\eta)={\rm rat.rk.}(v)=r,$$ in which case the weights $\alpha_1,\dots,\alpha_r$ are $\bQ$-linearly independent.
\bigskip

In the following, we also need a slight generalization of Definition \ref{d-quasimono}, namely instead of assuming that $(Y, E_1+\dots+E_r)$ is simple normal crossing at $\eta$, we assume
$$(\eta\in Y,E_1+\dots +E_r)\cong (\eta' \in Y',E'_1+\dots +E'_r)/G,$$
where $(\eta' \in Y',E'_1+\dots +E'_r)$ is a semi-local SNC pair and $ G$ is an abelian group, i.e. $(\eta\in Y,E_1+\dots +E_r)$ is toroidal around $\eta$. Denote the pull back of $E_i$ by $n_iE_i'$. Then for $\alpha=(\alpha_1,\dots, \alpha_r)\in \mathbb{R}^r_{\ge 0}$, we can define $v_{\alpha}$ to be the restriction of $v_{\alpha'}$ at $\eta'\in (Y',E'_1+\dots +E'_r)$ as in Definition \ref{d-quasimono}, where $\alpha'=(n_1\alpha_1,\dots, n_r\alpha_r)$ (the weights in $\alpha'$ are chosen so that $v_\alpha(E_i)=\alpha_i$). The definition is independent of the SNC abelian cover $(Y',E')$. We denote by $\QM_{\eta}(Y,E)$ all such valuations with $\alpha\in \bR^r_{\ge 0}$, and by $\QM^{\circ}_{\eta}(Y,E)$ with $\alpha\in \bR^r_{> 0}$. In general, given a model $(Y,E)$, we denote by $\QM(Y,E)\subseteq \Val_X$ the union of $\QM_\eta(Y,F)$ where $\eta\in Y$ and $F\subseteq E$ is a reduced divisor such that $(Y,F)$ is toroidal at $\eta$.

\begin{lem}\label{lem-valuationconvexity}
Let $\fa_{\bullet}$ be a graded sequence of ideals. Then the function 
$v\to v(\fa_{\bullet})$ is concave on $\QM_{\eta}(Y,E)$.
\end{lem}
\begin{proof}By passing to the abelian cover, we may assume $(Y,E)$ is log smooth at $\eta$. 
Assume $v_0$ and $v_1$ are in $\QM_{\eta}(Y,E)$, corresponding to $\alpha_0$ and $\alpha_1$ in $\bR^r_{\ge 0}$ and $v_t$ the valuation corresponding to $\alpha_t:=(1-t)\alpha_0+t\alpha_1$.
For any $f$ write $f=\sum c_{\beta}z^{\beta}$. Let $\beta_0$ satisfy 
$$c_{\beta_0}\ne 0\mbox{\ \ \ \ and \ \ \ \ }v_{t}(f)=\langle \alpha_t,\beta_0\rangle. $$
Then 
$$v_{t}(f)=(1-t)\langle\alpha_0,\beta_0\rangle+t\langle \alpha_1,\beta_0\rangle\ge (1-t)v_0(f)+tv_1(f).$$
Let $\fa$ be an ideal, and $v_t(\fa)=v_t(f)$ for some $f\in \fa$.
Then
\[
v_t(\fa)=v_t(f)\ge tv_0(f)+(1-t)v_1(f) \ge tv_0(\fa)+(1-t)v_1(\fa).
\]
Since $v(\fa_{\bullet})=\lim_{k\to \infty}\frac{1}{k}v(\fa_k)$ for any valuation $v$, the above inequality implies
$$v_t(\fa_{\bullet})\ge (1-t)v_0(\fa_{\bullet})+tv_1(\fa_{\bullet}).$$
\end{proof}




\begin{defn}
Let $(X,\Delta)$ be a log canonical pair, we use $\LC(X,\Delta)\subseteq \Val_X$ to denote the set of valuations which are lc places of $(X,\Delta)$.  It is the cone over the dual complex $\DMR(X,\Delta)$ defined in \cite[Theorem 28]{dFKX-dualcomplex}.
\end{defn}

\subsection{Normalized volume}

\begin{defn}\label{d-vol}
 Let $X$ be an $n$-dimensional normal variety. Let $x\in X$ be a closed point. Following \cite{ELS03}, we define
 the \emph{volume of a valuation} $v\in\Val_{X,x}$ as
 \[
  \vol_{X,x}(v)=\limsup_{k\to\infty}\frac{\ell(\cO_{X,x}/\fa_k(v))}{k^n/n!},
 \]
 where $\ell$ denotes the length of an Artinian module.
\end{defn}

By \cites{ELS03, LM09, Cut13}, the above limsup is actually a limit.

\begin{defn}[Log discrepancy]\label{d-logd}
 Let $(X,\Delta)$ be a log canonical pair. We define the \emph{log discrepancy
 function of valuations} $A_{X,\Delta}:\Val_X\to [0,+\infty]$
 in successive generality.
 \begin{enumerate}
  \item Let $\mu:Y\to X$ be a proper birational morphism from
  a normal variety $Y$. Let $E$ be a prime divisor
  on $Y$. Then we define $A_{X,\Delta}(\ord_E)=A_{X,\Delta}(E)$, i.e.
  \[
   A_{X,\Delta}(\ord_E):=1+\ord_E(K_Y-\mu^*(K_X+\Delta)).
  \]
  \item Let $(Y,E=\sum_{i=1}^r E_i)\to X$ be a log smooth model of $(X,\Delta)$. 
  Let $\eta$ be the generic point of a connected component of $E_1\cap\dots\cap E_r$. For any $v_\alpha\in \QM_\eta(Y,E)$ where $\alpha=(\alpha_1,\dots,\alpha_r)\in\bR_{\geq 0}^r\setminus\{0\}$, we define $A_{X,\Delta}(v_{\alpha})$ as
  \[
   A_{X,\Delta}(v_\alpha):=\sum_{i=1}^r \alpha_i A_{X,\Delta}(\ord_{E_i}).
  \]
  \item In \cite{JM12}, it was shown that there exists a retraction map 
  $$\rho_{Y,E}: \Val_X \to \QM(Y, E)$$ for any log smooth model $(Y, E)$ over $X$, such that it induces a
homeomorphism $\Val_X \to\varprojlim_{(Y,E)}\QM(Y, E)$. For any real valuation $v \in \Val_X$, we define
\begin{eqnarray*}\label{e-logdis}
A_{X,\Delta}(v):=\sup_{(Y,E)} A_{X,\Delta}(\rho_{Y,E}(v)),
\end{eqnarray*}
where $(Y, E)$ ranges over all log smooth models over $(X,\Delta)$. 
 \end{enumerate}
\end{defn}

The following invariant is introduced in \cite{Li18}.

\begin{defn}[\cite{Li18}]\label{d-normvol}
 Let $x\in (X,\Delta)$ be an $n$-dimensional klt singularity.  The \emph{normalized volume function of valuations} $\hvol_{(X,\Delta),x}:\Val_{X,x}\to \bR\cup \{+\infty\}$
 is defined as
 \[
  \hvol_{(X,\Delta),x}(v)=\begin{cases}
            A_{X,\Delta}(v)^n\cdot\vol_{X,x}(v), & \textrm{ if }A_{X,\Delta}(v)<+\infty;\\
            +\infty, & \textrm{ if }A_{X,\Delta}(v)=+\infty.
           \end{cases}
 \]
 \end{defn}
When the singularity $x\in (X,\Delta)$ is clear in the context, we will abbreviate the notation $\hvol_{(X,\Delta),x}$ as $\hvol_{(X,\Delta)}$ or $\hvol$. 
\cite{Li18} first proposed to study the minimizer of $\hvol_{(X,\Delta),x}$, and several deep conjectures were made there. Then the conjectural picture was supplemented in \cite{LX18}, where the relationship with K-stability was considered in full generality.
It is shown in \cite{Blum18} that a minimizer of $\hvol_{(X,\Delta),x}$ always exists. A priori, the definition includes to consider general valuations. However,  \cite{Xu20} proves that any minimizer is quasi-monomial. Therefore, we can restrict our studies in this paper to quasi-monomial valuations. For more background of the study on normalized volumes, see \cite{LLX20}.

\subsection{Koll\'ar components}\label{ss-kollar}

Recall that a divisor $E$ over a normal singularity $x\in X$ is said to be \emph{primitive} if there exists a projective birational morphism $\pi\colon Y\to X$ such that $Y$ is normal, $E$ is a prime divisor on $Y$ and $-E$ is a $\pi$-ample $\bQ$-Cartier divisor. We call $\pi\colon Y\to X$ \emph{the associated prime blowup} (it is uniquely determined by $E$).

\begin{lem} \label{lem:fg=primitive}
Let $E$ be a divisor over $x\in X=\Spec(R)$. Then $\gr_E R:=\gr_{\ord_E}R$ is finitely generated if and only if $E$ is primitive.
\end{lem}

\begin{proof}
Let $\fa_m=\fa_m(\ord_E)$ ($m\in\bN$). By lifting generators of $\gr_E R=\bigoplus_{m\in\bN} \fa_m/\fa_{m+1}$, we see that $\gr_E R$ is finitely generated if and only if $\cR=\bigoplus_{m\in\bN} \fa_m$ is finitely generated over $R$. We shall show that the latter is equivalent to the primitivity of $E$. If $E$ is primitive and $\pi\colon Y\to X$ is the associated prime blowup, then $\fa_m=\pi_*\cO_Y(-mE)$. As $-E$ is $\pi$-ample, there exists some $p\gg 0$ such that $\fa_{i+jp}=\fa_i\cdot \fa_p^j$ for all $i,j\in\bN$. It follows that $\cR$ is finitely generated. Conversely, if $\cR$ is finitely generated, then \cite{Ishii-prime-blowup}*{Corollary 3.3} implies that $E$ is primitive. 
\end{proof}

Let $\pi\colon Y\to X$ be the prime blowup of a primitive divisor $E$ over a normal singularity $x\in X$. Let $\fa_m:=\pi_*\cO_Y(-mE)$. Then by the Rees construction we get a $\bG_m$-equivariant degeneration 
\begin{equation}\label{eq-rees}
p\colon \cX\to \bA^1={\rm Spec}(\bk[t]) \mbox{ of $X$, \ \ \ where } \cX:=\Spec \bigoplus_{m\in\bZ} t^{-m} \fa_m.
\end{equation}
By the construction, there is a natural isomorphism $\iota\colon \cX\times_{\bA^1} (\bA^1\setminus\{0\})\cong X\times (\bA^1\setminus \{0\})$. 
So $\cX$ has a section $\gamma\colon \bA^1\to \cX$ given by the closure of $x\times (\bA^1\setminus\{0\})$. Let $X_0=p^{-1}(0)$ be the fiber over $0\in\bA^1$, and let $x_0=\gamma(0)$. Note that $X_0\cong\Spec(\bigoplus_{m\in\bN} \fa_m/\fa_{m+1})$. For any $\bQ$-divisor $\Delta$ on $X$, we denote by $\Delta_{\cX}$ the closure of $\Delta\times (\bA^1\setminus \{0\})$ on $\cX$ under the natural isomorphism $\iota$.

\medskip

When $x\in (X,\Delta)$ is a klt singularity, the following special class of primitive divisors plays a particularly important role (see e.g. \cite{Xu-kollar}).

\begin{defn}
Let $x\in (X,\Delta)$ be a klt singularity and let $E$ be a prime divisor over $X$. If there exists a proper birational morphism $\pi\colon Y\to X$ such that $E= \pi^{-1}(x)$ is the unique exceptional divisor, $(Y,E+\pi_*^{-1}\Delta)$ is plt and  $-(K_Y+E+\pi_*^{-1}\Delta)$ is $\pi$-ample, we call $E$ \emph{a Koll\'ar component} over $x\in (X,\Delta)$.
\end{defn}

The following simple criterion will be needed in order to verify that certain divisors over a given singularity are Koll\'ar components. 

\begin{lem} \label{lem:plt complement induce kc}
Let $x\in (X,\Delta)$ be a klt singularity and let $\pi\colon Y\to X$ be a proper birational morphism. Let $E\subseteq \pi^{-1}(x)$ be a prime exceptional divisor. Assume that there exists a $\bQ$-divisor $D\ge E+ \pi_*^{-1}\Delta$ such that $K_Y+D\sim_{\bQ,\pi} 0$ and $(Y,D)$ is plt. Then $E$ is a Koll\'ar component over $x\in (X,\Delta)$.
\end{lem}

\begin{proof}
Let $\Gamma = \pi_*D-\Delta\ge 0$. Since $K_Y+D\sim_{\bQ,\pi} 0$, we have $K_Y+D\sim_\bQ \pi^*(K_X+\Delta+\Gamma)$. Thus as $(Y,D)$ is plt we see that $\Gamma$ is a $\bQ$-complement of $x\in (X,\Delta)$ and $E$ is its unique lc place. By \cite{BCHM}*{Corollaries 1.4.3 and 1.3.1}, this implies that $E$ is primitive and we let $\varphi\colon Z\to X$ be the associated prime blowup. Then $(Z,E+\varphi_*^{-1}(\Delta+\Gamma))$ is lc since it is the crepant pullback of $(X,\Delta+\Gamma)$; it is also plt since it has only one lc place which is the divisor $E$. It follows that $(Z,E+\varphi_*^{-1}\Delta)$ is plt as well. As $-(K_X+E+\varphi_*^{-1}\Delta)\sim_{\bQ,\varphi} -A_{X,\Delta}(E)\cdot E$ is $\varphi$-ample by construction, we conclude that $E$ is a Koll\'ar component.
\end{proof}

\begin{lem}\label{lem-kollardegenerate}
If $E$ is a Koll\'ar component over a klt singularity $x\in (X,\Delta)$, and let 
$$p\colon (\cX,\Delta_{\cX})\to \bA^1={\rm Spec}(\bk[t])$$ be the $\bG_m$-equivariant family induced by the Rees construction as in \eqref{eq-rees}. Then $(\cX,\Delta_{\cX}+p^{-1}(0))$ is plt. 

Moreover, if $E$ is an lc place of a $\bQ$-complement $\Gamma$, and denote by $\Gamma_{\cX}$ the closure of $\Gamma\times (\bA^1\setminus \{0\})$ on $\cX$,
 then $(\cX,\Delta_{\cX}+\Gamma_{\cX}+p^{-1}(0))$ is lc.
\end{lem}
\begin{proof}
The first part was shown in \cite[Section 2.4]{LX20}. 

From the assumption, we know that for any $0\le t<1$, $E$ is also a Koll\'ar component over the klt pair 
$(X,\Delta+t\Gamma)$. Thus from the first part, we know $(\cX,\Delta_{\cX}+t\cdot \Gamma_{\cX}+p^{-1}(0))$ is plt for any $0\le t<1$. So 
$(\cX,\Delta_{\cX}+\Gamma_{\cX}+p^{-1}(0))$ is lc. 
\end{proof}
 
We also have the following result, which implies Theorem \ref{thm-localfinite} in the case when $v$ is divisorial. 

\begin{thm}
Let $x\in (X,\Delta)$ be a klt singularity. Let $v$ be a minimizer of the normalized volume function $\hvol_{X,\Delta}$. Assume that $v$ is divisorial. Then $v=\lambda\cdot \ord_E$ for some Koll\'ar component $E$.
\end{thm}

\begin{proof}
See \cite{LX20}*{Theorem C} or \cite{Blum18}*{Proposition 4.9}.
\end{proof}

Since in general, a minimizer of $\hvol_{X,\Delta}$ could be a quasi-monomial valuation with rational rank larger than one, we have to adapt the study for higher rank valuations.
For this purpose, we will introduce the concept of Koll\'ar models (see Definition \ref{defn:qdlt Fano model}).

\section{Lc places of special $\bQ$-complements}

The concept of ($\bQ$-)complements was first introduced in \cite{Shokurov92}. However, it is known that unlike the divisorial case, an arbitrary lc place $v$ of a $\bQ$-complement with rational rank larger than one does not necessarily have a finitely generated associated graded algebra (see e.g. \cite{AZ-index2}*{Theorem 4.16} and \cite{LXZ-HRFG}*{Theorem 6.1}).
We need to look at a smaller class of $\bQ$-complements, which we call special $\bQ$-complements. The goal of this section is then two-folds. First we relate minimizers of the normalized volume function to lc places of special $\bQ$-complement. More precisely, we will prove that for any klt singularity $x\in (X,\Delta)$, any minimizer of the normalized volume function $\hvol$ is a monomial lc place of a special $\bQ$-complement $\Gamma$ with respect to some log smooth model $(Y,E)$ (see Lemma \ref{lem:minimizer's complement}). Next we will introduce the concept of Koll\'ar models as a higher rank analog of Koll\'ar components. As a crucial step towards the finite generation statement, we show that any monomial lc place of a special $\bQ$-complement with respect to some log smooth model can also be realized on a Koll\'ar model (Theorem \ref{thm:construction of qdlt fano type model}). In the next section, we will prove finite generation by analyzing degenerations of Koll\'ar models.

\subsection{Special $\bQ$-complements} \label{ss-specialcomplement}

First we give the definition of special $\bQ$-complements, which is analogous to \cite{LXZ-HRFG}*{Definition 3.3}.

\begin{defn}\label{d-specialcomplement}
Let $(X,\Delta)$ be klt and projective over a quasi-projective variety $U$ such that $-K_X-\Delta$ is ample over $U$.  
A $\bQ$-complement $\Gamma$ of $(X,\Delta)$ will be called \emph{special with respect to a log smooth} (\emph{resp. toroidal model}) $\pi\colon(Y,E)\to (X,\Delta)$ if $\pi_*^{-1}\Gamma\ge G$ for some effective $\bQ$-divisor $G$ on $Y$ ample over $U$ whose support does not contain any stratum of $(Y,E)$.

Any valuation $$v\in \QM(Y,E)\cap \LC(X,\Delta+\Gamma)$$ is called a \emph{monomial lc place} of the special $\bQ$-complement $\Gamma$ with respect to $(Y,E)$. We denote by $\LC(\Gamma;Y,E):=\QM(Y,E)\cap \LC(X,\Delta+\Gamma)$ the corresponding set of monomial lc places. 

When $U=\Spec(\bk)$  this reduces to \cite{LXZ-HRFG}*{Definition 3.3}. In the local case we are mainly interested in here, i.e. when $x\in(X,\Delta)$ is a klt singularity and $U=X$, we further denote $\LC_x(\Gamma;Y,E)$ to be the set of valuations in $\Val_{X,x}\cap \LC(\Gamma;Y,E)$.
\end{defn}

Now let us show that minimizers of the normalized volume function are monomial lc places of special $\bQ$-complements (Lemma \ref{lem-localmonomial}). Once this is done, we will focus on showing that any monomial lc place of special $\bQ$-complements has a finitely generated graded ring.

We start with the following auxiliary result, which easily implies Lemma \ref{lem-localmonomial}. It can be viewed as the local analog of \cite{LXZ-HRFG}*{Lemma 3.1}.

\begin{lem} \label{lem:minimizer's complement}
Let $x\in (X,\Delta)$ be a klt singularity of dimension $n$ and let $v_0\in \Val_{X,x}$ be a minimizer of the normalized volume function $\hvol_{(X,\Delta)}$. Then for any effective $\bQ$-Cartier $\bQ$-divisor $D$ such that $v_0(D)<\frac{A_{X,\Delta}(v_0)}{n}$, there exists some $\bQ$-complement $\Gamma$ of $(X,\Delta)$ such that $\Gamma\ge D$ and $v_0$ is an lc place of $(X,\Delta+\Gamma)$.
\end{lem}

As a preliminary step, we prove:

\begin{lem} \label{lem:minimizer computes lct}
Let $x\in (X=\Spec(R),\Delta)$ be a klt singularity of dimension $n$ and let $v_0\in \Val_{X,x}$ be a minimizer of the normalized volume function $\hvol_{(X,\Delta)}$. Let $D$ be an effective $\bQ$-Cartier $\bQ$-divisor on $X$ such that $v_0(D)<\frac{A_{X,\Delta}(v_0)}{n}$. Then
\[
\lct_x(X,\Delta+D;\fab(v_0))=A_{X,\Delta+D}(v_0),
\]
where $\fab(v_0):=\{\fa_{m}(v_0)\}_{m\in \bN}$ is the graded sequence of valuation ideals. In particular, the log canonical threshold $\lct_x(X,\Delta+D;\fab(v_0))$ is computed by $v_0$.
\end{lem}

\begin{proof}
If $v_0(D)=0$, i.e. $x\not\in\Supp(D)$, the result simply follows from \cite{Blum18}*{Lemma 4.7}, hence we may assume that $v_0(D)>0$. Up to rescaling the valuation $v_0$, we may assume that $A_{X,\Delta}(v_0)=1$. Let $c=A_{X,\Delta+D}(v_0)$ and let $\fab=\fab(v_0)$. Clearly $v_0(\fab)=1$ (see \cite{Blum18}*{Lemma 3.5}) and
\[
\lct_x(X,\Delta+D;\fab)\le \frac{A_{X,\Delta+D}(v_0)}{v_0(\fab)} = c.
\]
Suppose that the inequality is strict, then there exists some valuation $v\in \Val_{X,x}$ such that $A_{X,\Delta+D}(v)<c\cdot v(\fab)$. Equivalently, $A_{X,\Delta}(v)<v(D)+c\cdot v(\fab)$. Choose some $0<\varepsilon\ll 1$ so that
\begin{equation} \label{eq:lct smaller than expected}
    (1+\varepsilon)A_{X,\Delta}(v)< v(D)+c\cdot v(\fab).
\end{equation}
We will derive a contradiction from here. 

First let us recall some definitions and results from \cite{XZ-uniqueness}*{Section 3}. Let $N_m=\dim_{\bk}(R/\fa_m)$ and let 
\begin{eqnarray}\label{e-wm}
w_m=\sum_{i=0}^{m-1} i\cdot \dim_{\bk}(\fa_i/\fa_{i+1}).
\end{eqnarray}
The valuation $v_0$ (resp. $v$) induces a natural filtration on $R$, which restricts to a filtration $\cF_{v_0}$ (resp. $\cF_v$) on $R/\fa_m$ with $$\cF_v^{\lambda}(R/\fa_m):=(\fa_{\lambda }(v)+\fa_m)/\fa_m \qquad \mbox{and}\qquad \cF_{v_0}^{\lambda}(R/\fa_m):=(\fa_{\lambda}(v_0)+\fa_m)/\fa_m.$$
A basis of $R/\fa_m$ is said to be compatible with a filtration $\cF$ of $R/\fa_m$ if every filtered subspace of $\cF$ is spanned by some of the basis elements. Any $s_1,\dots,s_{N_m}\in R$ that restrict to form a basis of $R/\fa_m$ that is compatible with $\cF_{v_0}$ is called an $m$-basis (with respect to $v_0$). If in addition it is also compatible with $\cF_v$ and whenever $\bar{s}_i\in \cF_v^\lambda(R/\fa_m)$ we get $s_i\in \fa_\lambda(v)$, we call it an $(m,v)$-basis. Following \cite{XZ-uniqueness}*{Section 3.1}, we define $$\tS_m(v_0;v)=\sum^{N_m}_{i=1} \lfloor v(s_i) \rfloor, \mbox{ where $s_1,\dots,s_{N_m}\in R$ is any $(m,v)$-basis with respect to $v_0$}.$$ 
 Note that \eqref{e-wm} implies that $\tS_m(v_0;v_0)=w_m$. We also set $$S(v_0;v)=\lim_{m\to \infty} \frac{\tS_m(v_0;v)}{w_m}$$ (the limit exists by \cite{XZ-uniqueness}*{Lemma 3.3}). By \cite{XZ-uniqueness}*{Theorem 3.10}, the minimizing valuation $v_0$ is K-semistable (\cite{XZ-uniqueness}*{Definition 3.4}). In particular we get 
\begin{equation} \label{eq:S(v_0;v)<=A(v)}
    S(v_0;v)\le A_{X,\Delta}(v).
\end{equation}
A $\bQ$-divisor of the form 
\[
\Gamma=\frac{1}{w_m}\sum_{i=1}^{N_m} {\rm div}(s_i)
\]
is called an \emph{$m$-basis type divisor with respect to $v_0$} if $s_1,\dots,s_{N_m}\in R$ is an $m$-basis with respect to $v_0$. Note that as the $m$-basis varies, $v_0(\Gamma)$ is independent of the $m$-basis $s_1,\dots,s_{N_m}$, while $v(\Gamma)$ achieves its maximum exactly when $s_1,\dots,s_{N_m}$ is an $(m,v)$-basis. Thus we have
\[
v_0(\Gamma)\ge \frac{\tS_m(v_0;v)}{w_m} = 1\quad \mbox{and}\quad 
v(\Gamma)\le \frac{\tS_m(v_0;v)+N_m}{w_m}.
\]
As $\frac{N_m}{w_m}$ grows like $O(\frac{1}{m})$ (see \cite[Proof of Lemma 3.3]{XZ-uniqueness}), we deduce from \eqref{eq:S(v_0;v)<=A(v)} that as $m\gg 0$ we have
\[
v(\Gamma)\le (1+\varepsilon)S(v_0;v) \le (1+\varepsilon)A_{X,\Delta}(v)
\]
for any $m$-basis type divisor $\Gamma$.

To derive a contradiction, we will construct an $m$-basis type divisor $\Gamma$ with respect to $v_0$ (for some $m\gg 0$) such that $\Gamma\ge D$. To see why this is enough, note that
\[
v_0(\Gamma-D)=v_0(\Gamma)-v_0(D)\ge 1-v_0(D)=A_{X,\Delta}(v_0)-v_0(D)=c,
\]
hence 
$$v(D)+c\cdot v(\fab)\le v(D)+v(\Gamma-D)= v(\Gamma)\le (1+\varepsilon)A_{X,\Delta}(v),$$ contradicting \eqref{eq:lct smaller than expected}.

The construction of $\Gamma$ proceeds as follows. Let $r>0$ be a sufficiently divisible integer such that $rD={\rm div}(f)$ is Cartier. Consider the filtration $\cF$ (indexed by $\bN$) on $R/\fa_m$ induced by the decreasing sequence of ideals $(f^j)_{i\in\bN}$. In more explicit terms, we have an isomorphism
\begin{equation} \label{eq:F_D}
    R/\fa_{m-j\cdot v_0(f)}(v_0)\stackrel{\cdot f^j}{\longrightarrow} \cF^j(R/\fa_m)
\end{equation}
given by multiplication by $f^j$. Choose a basis $\bar{s}_1,\dots,\bar{s}_{N_m}$ of $R/\fa_m$ that is compatible with both $\cF_{v_0}$ and $\cF$ (such a compatible basis exists by \cite{AZ-index2}*{Lemma 3.1}). We may lift them to an $m$-basis $s_1,\dots,s_{N_m}\in R$ such that $f^{-j}s_i\in R$ whenever $\bar{s}_i\in \cF^j(R/\fa_m)$ (this is possible by \cite{XZ-uniqueness}*{Lemma 2.1}). This gives an $m$-basis type divisor $\Gamma=\frac{1}{w_m} \sum_{i=1}^{N_m} {\rm div}(s_i)$. Let us check that $\Gamma\ge D$ when $m\gg 0$. Indeed, by construction we know that $w_m \cdot \Gamma\ge \mu_m\cdot {\rm div}(f)$ where 
\[
\mu_m=\sum_{j=1}^{+\infty} \dim \cF^j(R/\fa_m).
\]
Using the isomorphism \eqref{eq:F_D}, we get
\[\mu_m=\sum_{j=1}^{\lfloor \frac{m}{v_0(f)}\rfloor} \dim (R/\fa_{m-j\cdot v_0(f)}(v_0)).
\]
Keeping in mind that $\lim_{\lambda\to +\infty} \frac{\dim(R/\fa_\lambda(v_0))}{\lambda^n/n!}=\vol_{X,x}(v_0)$, we obtain the asymptotic expression
\[
\lim_{m\to\infty}\frac{\mu_m}{m^{n+1}/(n+1)!} = \frac{\vol_{X,x}(v_0)}{v_0(f)}.
\]
Similarly, by \eqref{e-wm} 
$$w_m=m\cdot \dim(R/\fa_m)-\sum_{j=1}^m \dim (R/\fa_j),$$ we see that
\[
\lim_{m\to\infty}\frac{w_m}{m^{n+1}/(n+1)!} = n\cdot \vol_{X,x}(v_0).
\]
Putting the asymptotic together we deduce that $$\lim_{m\to\infty} \frac{r\mu_m}{w_m}=\frac{1}{n\cdot v_0(D)};$$ by assumption, the right hand side is $>\frac{1}{A_{X,\Delta}(v_0)}=1$. Hence if $m\gg 0$ we have $\frac{\mu_m}{w_m}>\frac{1}{r}$ which translates to $\Gamma\ge D$ as desired. By the previous discussions, this concludes the proof.
\end{proof}

\begin{proof}[Proof of Lemma \ref{lem:minimizer's complement}]
This essentially follows from \cite{LX18}*{Theorem 2.50 and Remark 2.52} (c.f. the proof of \cite{LXZ-HRFG}*{Lemma 3.1}). We provide the argument for the reader's convenience. 

By ACC of log canonical thresholds \cite{HMX-ACC}, there exists a constant $\varepsilon\in (0,1)$ depending only on $\dim X$ and the coefficients of $\Delta+D$ satisfying the following: for any proper birational morphism $\pi\colon \tX\to X$ where $\tX$ is $\bQ$-factorial and any reduced divisor $G$ on $\tX$ such that $$(\tX,\pi_*^{-1}(\Delta+D)+(1-\varepsilon)G)$$ is lc, we have $(\tX,\pi_*^{-1}(\Delta+D)+G)$ is lc as well. 

By \cite{Xu20}, the minimizer $v_0$ is quasi-monomial. Let $r$ be the rational rank of $v_0$. Let $\pi\colon Y\to (X,\Delta+D)$ be a log resolution such that $v_0$ lies in the interior of $\QM_\eta(Y,E)$ for some SNC divisor $E=E_1+\dots+E_r$ on $Y$ with $\pi(E)=\{x\}$, where $\eta=c_Y(v_0)$ is the generic point of a connected component of $\bigcap_{i=1}^r E_i$. Let $\fab=\fab(v_0)$ and $c=A_{X,\Delta+D}(v_0)$ as in Lemma \ref{lem:minimizer computes lct}. Since on $\QM_\eta(Y,E)$, $A_{X,\Delta+D}(w)$ is linear in $w$ as $\pi$ is a log resolution, and $w\mapsto w(\fab^c)$ is concave by Lemma \ref{lem-valuationconvexity}, we conclude that the log discrepancy function 
$$w\mapsto A_{X,\Delta+D+\fab^c}(w)=A_{X,\Delta+D}(w)-w(\fab^c)$$ is convex on $\QM_\eta(Y,E)$. 
 In particular it is  Lipschitz on $\QM_\eta(Y,E)$, hence there exists some constant $C>0$ such that 
\[
|A_{X,\Delta+D+\fab^c}(w)-A_{X,\Delta+D+\fab^c}(v_0)|\le C|w-v_0|
\]
for all $w\in \QM_\eta(Y,E)$. Here $|w-v_0|$ denotes the Euclidean distance of $w$ and $v_0$ in $\QM_\eta(Y,E)\cong \bR_{\geq 0}^r$. By \cite{LX18}*{Lemma 2.7}, there exist divisorial valuations $$v_1,\dots,v_r\in \QM_\eta(Y,E)$$ and positive integers $q_1,\dots,q_r$ such that 
\begin{itemize}
    \item $v_0$ is in the convex cone generated by $v_1,\dots,v_r$,
    \item for all $i=1,\dots,r$, the valuation $q_i v_i$ is $\bZ$-valued and has the form $\ord_{F_i}$ for some divisor $F_i$ over $X$, and
    \item $|v_i-v_0|<\frac{\varepsilon}{Cq_i}$ for all $i=1,\dots,r$. 
\end{itemize}
As $A_{X,\Delta+D+\fab^c}(v_0)=0$ by Lemma \ref{lem:minimizer computes lct}, the above conditions imply that 
\begin{equation} \label{eq:A(F_i)<epsilon}
    A_{X,\Delta+D+\fab^c}(F_i)=q_i A_{X,\Delta+D+\fab^c}(v_i)\le q_i\cdot C|v_i-v_0|<\varepsilon.
\end{equation}
On the other hand, $(X,\Delta+D+\fab^c)$ is lc by Lemma \ref{lem:minimizer computes lct}, thus we may choose some $0<\varepsilon_0\ll 1$ and some $m\gg 0$ such that $(X,\Delta+D+(c-\varepsilon_0)\fa_m)$ is klt and 
\[
A_{X,\Delta+D+(c-\varepsilon_0)\fa_m}(F_i)<\varepsilon.
\]
By \cite{BCHM}*{Corollary 1.4.3}, there exists a projective birational morphism $\pi_1\colon X_1\to X$ that extracts exactly the divisors $F_1,\dots,F_r$, and we may assume $X_1$ is $\bQ$-factorial. Note that $X_1$ is of Fano type over $X$. Running a $-(K_{X_1}+\Delta_1+D_1+\sum_{i=1}^r F_i)$-MMP over $X$ (where $\Delta_1,D_1$ are the birational transforms of $\Delta,D$), we obtain a model $g\colon X_1\dashrightarrow \tX$ over $X$ such that $-(K_{\tX}+g_*(\Delta_1+D_1+\sum_{i=1}^r F_i))$ is nef. The equality \eqref{eq:A(F_i)<epsilon} implies that $$(\tX,g_*(\Delta_1+D_1+(1-\varepsilon)\sum_{i=1}^r F_i))$$ is lc, thus by our choice of $\varepsilon$ we deduce that $(\tX,g_*(\Delta_1+D_1+\sum_{i=1}^r F_i))$ is lc. But this implies that the MMP process $X_1\dashrightarrow \tX$ is birational on any lc center of $(X_1,\Delta_1+D_1+\sum_{i=1}^r F_i)$: otherwise we will have a divisor $F$ over this center with
\[
A_{\tX,g_*(\Delta_1+D_1+\sum_{i=1}^r F_i)}(F)<A_{X_1,\Delta_1+D_1+\sum_{i=1}^r F_i}(F)=0,
\]
a contradiction. In particular, it does not contract any of the divisors $F_i$ and we may replace $X_1$ by $\tX$ and assume that $(X_1,\Delta_1+D_1+\sum_{i=1}^r F_i)$ is lc and $-(K_{X_1}+\Delta_1+D_1+\sum_{i=1}^r F_i)$ is nef. As $X_1$ is of Fano type over $X$, it follows that $-(K_{X_1}+\Delta_1+D_1+\sum_{i=1}^r F_i)$ is base point free. Hence by Bertini's theorem we can choose some effective $\bQ$-divisor $\Gamma_1$ such that
\[
K_{X_1}+\Delta_1+D_1+\Gamma_1+\sum_{i=1}^r F_i\sim_{\bQ,\pi_1} 0
\]
and $(X_1,\Delta_1+D_1+\Gamma_1+\sum_{i=1}^r F_i)$ is still lc. Let $\Gamma=\pi_{1*}(D_1+\Gamma_1)$, then $\Gamma$ is a $\bQ$-complement of $(X,\Delta)$ that has all $F_1,\dots,F_r$ as lc places. Since $v_0$ is in their convex hull, it is also an lc place of $(X,\Delta+\Gamma)$. By construction, $\Gamma\ge \pi_{1*}D_1=D$ and we are done.
\end{proof}

We now show that any minimizer of the normalized volume function is a monomial lc place of some special $\bQ$-complement.

\begin{lem}\label{lem-localmonomial}
Let $x\in (X={\rm Spec}(R),\Delta)$ be a klt singularity and let $v_0$ be the minimizer of the normalized volume function. Then there exists a log smooth model $\pi\colon (Y,E)\to (X,\Delta)$ such that $v_0$ is a monomial lc place of some special $\bQ$-complement $\Gamma$ with respect to $(Y,E)$, i.e. $v_0\in \LC(\Gamma; Y,E)$.
\end{lem}

\begin{proof}
Let $\pi\colon (Y,E)\to (X,\Delta)$ be a log smooth model such that $v_0\in \QM(Y,E)$ and $-F$ is $\pi$-ample for some effective divisor $F$ supported on the exceptional locus. Let $r>0$ be a sufficiently large integer and let $f\in \pi_*\cO_Y(-rF)$ be a general element. Let $D_0=\{f=0\}$, then the birational transform $\pi_*^{-1}D_0\sim_{\bQ,\pi} -rF$ is $\pi$-ample and does not contain any stratum of $(Y,E)$ by Bertini's theorem. Since $X$ is affine, the $\pi$-ample divisor $\pi_*^{-1}D_0$ is in fact ample. Let $D=tD_0$ for some $0<t\ll 1$ such that $v_0(D)<\frac{A_{X,\Delta}(v_0)}{\dim X}$. By Lemma \ref{lem:minimizer's complement}, there exists some $\bQ$-complement $\Gamma$ of $(X,\Delta)$ such that $\Gamma\ge D$ and $v_0$ is an lc place of $(X,\Delta+\Gamma)$. By construction, $\Gamma$ is a special $\bQ$-complement with respect to $(Y,E)$. This completes the proof. 
\end{proof}

\subsection{Koll\'ar models}

Our next goal is to reduce the study of special $\bQ$-complements (and their monomial lc places) to the study of Koll\'ar models. In this subsection, we first introduce the concept of Koll\'ar models and discuss some of their general properties.

\begin{defn} \label{defn:qdlt Fano model}
Let $(X,\Delta)$ be a klt pair that's projective over a quasi-projective variety $U$ such that $-K_X-\Delta$ is ample over $U$.
A model $\pi\colon(Y,E)\to (X,\Delta)$ is said to be \emph{of qdlt Fano type} if there exists an effective $\bQ$-divisor $D\ge \pi_*^{-1}\Delta$ on $Y$ such that $\lfloor D \rfloor = E$, $(Y,D)$ is qdlt, and $-(K_Y+D)$ is ample over $U$.
\end{defn}

Note that these conditions imply that
there exists a $\bQ$-complement $\Gamma$ of $(X,\Delta)$ such that $Y$ is a qdlt modification (in the sense of e.g. \cite{Kol13}*{Theorem 1.34}) of $(X,\Delta+\Gamma)$ and $\pi_*^{-1}\Gamma$ contains an ample $\bQ$-subdivisor.

The following proposition demonstrates some of the flexibility in the construction of qdlt Fano type models.

\begin{prop} \label{prop:modify qdlt Fano} 
Let $(X,\Delta)$ be a klt pair projective over $U$ such that $-K_X-\Delta$ is ample over $U$. Let $\pi\colon(Y,E)\to (X,\Delta)$ be a model of qdlt Fano type. Then:
\begin{enumerate}
    \item For any reduced divisor $E_0\le E$, the model $(Y,E_0)\to (X,\Delta)$ is of qdlt Fano type.
    \item For any birational contraction $\varphi\colon Y\dashrightarrow Y'$ over $X$ that is an isomorphism at the generic point of every stratum of $E$, the induced model $\pi'\colon (Y',E'=\varphi_*E)\to (X,\Delta)$ is of qdlt Fano type.
\end{enumerate}
\end{prop}
\begin{proof}
Let $D$ be the $\bQ$-divisor in Definition \ref{defn:qdlt Fano model}. We can choose some $\bQ$-divisor $\Delta_0$ such that $\lfloor \Delta_0 \rfloor =E_0$ and $\Delta_0\ge \pi_*^{-1}\Delta$. By Lemma \ref{lem:perturb qdlt boundary} we get a $\bQ$-divisor $D_0\ge\Delta_0\ge \pi_*^{-1}\Delta$ such that $\lfloor D_0 \rfloor = \lfloor \Delta_0 \rfloor =E_0$, $(Y,D_0)$ is qdlt and $-(K_Y+D_0)$ is ample over $U$. This implies (1).

Let $G'$ be an ample divisor on $Y'$ that is in a general position and let $G=\varphi_*^{-1}G'$. Then by the assumption on $\varphi$ and Bertini's theorem we know that $\Supp(G)$ does not contain any stratum of $E$. By Lemma \ref{lem:perturb qdlt boundary}, there exists some $0<\varepsilon\ll 1$ and some $D_1\ge \pi_*^{-1}\Delta+\varepsilon G$ such that $\lfloor D_1 \rfloor =E$, $(Y,D_1)$ is qdlt and
\[
-(K_Y+D_1)\sim_{\bQ} 0\ /U. 
\]
In fact, Lemma \ref{lem:perturb qdlt boundary} gives some $D_1$ as above such that $-(K_Y+D_1)$ is ample; but then by Bertini's theorem we can add a general ample divisor to $D_1$ so that $-(K_Y+D_1)\sim_\bQ 0\ /U$. 

Let $D'_1=\varphi_* D_1$. Since $\varphi$ is an isomorphism at the generic points of all the strata of $E$, we see that $(Y',D'_1)$ is also qdlt and $-(K_{Y'}+D'_1)\sim_\bQ 0\ /U$. Since $D'_1\ge \varepsilon G'$, the $\bQ$-divisor $D'=D'_1-\varepsilon G'$ makes $(Y',E')$ a qdlt Fano type model over $(X,\Delta)$, as $\lfloor D' \rfloor =E'$ and
\[
-(K_{Y'}+D')\sim_{\bQ,U} \varepsilon G' 
\]
is ample by construction. This gives (2).
\end{proof}

\begin{defn} \label{defn:km}
Let $x\in (X,\Delta)$ be a klt singularity. A \emph{Koll\'ar model} of $x\in (X,\Delta)$ is a model $\pi\colon (Y,E)\to (X,\Delta)$ such that $\pi$ is an isomorphism away from $\{x\}$, $E=\Ex(\pi)$, $(Y,\pi_*^{-1}\Delta+E)$ is qdlt and $-(K_Y+\pi_*^{-1}\Delta+E)$ is ample.
\end{defn}

Clearly a Koll\'ar model is also of qdlt Fano type (with $U=X$). The terminology comes from the following basic observations.

\begin{lem} \label{lem:kc induce km} 
Let $x\in (X,\Delta)$ be a klt singularity. Then the plt blowup $\pi\colon Y\to X$ of any Koll\'ar component $E$ is a Koll\'ar model $(Y,E)$ over $(X,\Delta)$. Conversely, if $\pi\colon(Y,E)\to (X,\Delta)$ is a Koll\'ar model, then every irreducible component of $E$ is a Koll\'ar component over $x\in (X,\Delta)$.
\end{lem}

\begin{proof}
The first part is clear from the definition. For the second part, suppose that $\pi\colon(Y,E)\to (X,\Delta)$ is a Koll\'ar model. Let $E_1$ be an irreducible component of $E$. By Lemma \ref{lem:perturb qdlt boundary}, there exists some effective $\bQ$-divisor $D_1\ge \pi_*^{-1}\Delta+E_1$ on $Y$ such that  $-(K_Y+D_1)$ is ample, $\lfloor D_1 \rfloor = E_1$ and $(Y,D_1)$ is qdlt. Since $\lfloor D_1 \rfloor$ is irreducible, this indeed implies the pair $(Y,D_1)$ is plt. Let $m>0$ be some sufficiently divisible integer and let $G\in \frac{1}{m}|-m(K_Y+D_1)|$ be a general member. Then $K_Y+D_1+G\sim_\bQ 0$ and $(Y,D_1+G)$ remains plt by Bertini's theorem. By Lemma \ref{lem:plt complement induce kc}, we conclude that $E_1$ is a Koll\'ar component.
\end{proof}

Over a klt singularity, the two types of models introduced above are closely related. In fact, the following proposition shows that Koll\'ar models are nothing but the ample models of qdlt Fano type models.

\begin{prop} \label{prop:qdlt to km}
Let $x\in (X,\Delta)$ be a klt singularity, and let $\pi\colon(Y,E)\to (X,\Delta)$ be a model of qdlt Fano type such that $\pi(E)=\{x\}$. Let $g\colon Y\dashrightarrow Y'$ be the ample model of $-(K_Y+\pi_*^{-1}\Delta+E)$ over $X$. Then $g$ is an isomorphism at the generic point of every stratum of $E$, and the induced morphism $\pi'\colon (Y',E'=g_*E)\to (X,\Delta)$ is a Koll\'ar model.
\end{prop}

We remark that here the ample model exists since $Y$ is of Fano type over $X$.

\begin{proof}
Let us first show that $g\colon Y\dashrightarrow Y'$ is an isomorphism at the generic point of every stratum of $E$. Let $\varphi \colon \tY\to Y$ be a small $\bQ$-factorial modification, which exists by \cite{BCHM}*{Corollary 1.4.3}. Note that $\varphi$ is an isomorphism at the locus where $Y$ is $\bQ$-factorial; in particular, this includes the generic points of all the strata of $E$. By Proposition \ref{prop:modify qdlt Fano}, we deduce that the model $(\tY,\tE=\varphi_*E)$ is of qdlt Fano type. Since a small birational map such as $\varphi$ does not change the ample model in question, we may replace $(Y,E)$ by $(\tY,\tE)$ and assume that $Y$ is $\bQ$-factorial.

Next we run the MMP $h\colon Y\dashrightarrow \overline{Y}$ over $X$ with respect to $L:=-(K_Y+\pi_*^{-1}\Delta+E)$. We need to show that the MMP is an isomorphism at the generic points of all strata of $E$. To see this, note that if $D\ge E+\pi_*^{-1}\Delta$ is the $\bQ$-divisor from Definition \ref{defn:qdlt Fano model} that makes $(Y,E)$ of qdlt Fano type, then
\begin{equation} \label{eq:ample+eff}
    L\sim_{\bQ,\pi} -(K_Y+D)+(D-E-\pi_*^{-1}\Delta) = ({\rm ample})+(D-E-\pi_*^{-1}\Delta)
\end{equation}
and $\Supp(D-E-\pi_*^{-1}\Delta)$ does not contain any stratum of $E$ (otherwise $(Y,D)$ would not be qdlt). Thus the MMP is an isomorphism at the generic points of all strata of $E$. By Proposition \ref{prop:modify qdlt Fano}, the model $(\overline{Y},\overline{E}=h_*E)$ is of qdlt Fano type. Since the MMP does not change the ample model of $L$, after replacing $(Y,E)$ by $(\overline{Y},\overline{E})$ (and modify the corresponding divisor $D$) we may further assume that $L$ is nef.

The ample model $g\colon Y\to Y'$ of $L$  contracts all subvarieties $W\subseteq Y$ such that $L|_W$ is not big. By \eqref{eq:ample+eff} again, we see that the contracted locus is contained in $\Supp(D-E-\pi_*^{-1}\Delta)$, which does not contain any stratum of $E$. Therefore, we conclude that $g$ is an isomorphism at the generic point of every stratum of $E$. 

By Proposition \ref{prop:modify qdlt Fano}, it then follows that $\pi'\colon (Y',E')\to (X,\Delta)$ is of qdlt Fano type. To show that $\pi'$ is in fact a Koll\'ar model, we first observe that $E'\subseteq \Ex(\pi')$ as $\pi'(E')=\pi(E)=\{x\}$. On the other hand,
\[
-(K_{X'}+\pi'^{-1}_*\Delta+E')\sim_{\bQ,\pi'} -\sum_{i=1}^r A_{X,\Delta}(E'_i)\cdot  E'_i 
\]
is ample by construction, where $E'_1,\dots,E'_r$ are the components of $E'$. Since $A_{X,\Delta}(E'_i)>0$, this implies that $\Ex(\pi')$ is necessarily contained in $E'$, thus $E=\Ex(\pi')$. As $\pi'(E')=\{x\}$, we further deduce that $\pi'$ is an isomorphism away from $\{x\}$. This finishes the proof.
\end{proof}

The class of Koll\'ar models (or more generally, qdlt Fano type models) enjoys many stronger properties that qdlt modifications of a general $\bQ$-complement do not always have. The following result gives one example, which illuminates our later discussion.

\begin{prop} \label{prop:km DMR is simplex}
Let $(Y,E=\sum^r_{i=1}E_i)\to (X,\Delta)$ be a Koll\'ar model over a klt singularity $x\in (X,\Delta)$. Then the dual complex $\cD(E)$ is a simplex, i.e. the intersections of any subsets of $\{E_1,\dots,E_r\}$ is non-empty and irreducible.
\end{prop}

This follows from the fact there is a unique minimal lc center for $(Y,D+E)$ where $D$ is the divisor in the definition of Koll\'ar model.  
See \cite[Proposition 25]{dFKX-dualcomplex}. Later we need the following slightly stronger result.

\begin{lem} \label{lem:lc center of Fano pair}
Let $(X,D)$ be an lc pair and let $f\colon X\to U$ be a projective morphism such that $f_*\cO_X=\cO_U$ and $-(K_X+D)$ is $f$-ample. Assume that there exists an effective $\bQ$-divisor $\Delta$ on $X$ such that $(X,\Delta)$ is klt. Then for any $u\in U$, the set of lc centers of $(X,D)$ that intersects $f^{-1}(u)$ has a unique minimal element with respect to inclusion. In particular, for any lc center $W$ of $(X,D)$, the map $f|_W\colon W\to U$ has connected fibers.
\end{lem}

\begin{proof}
This is a special case of \cite{Ambro-quasilog}*{Theorem 6.6} and in fact the statement holds without assuming the existence of a klt boundary $\Delta$. We provide a direct proof in our special case for the reader's convenience. The statement is local on $U$, thus after shrinking around $u\in U$ we may assume that $L:=-(K_X+D)$ is ample. Since any intersection of lc centers is a union of lc centers \cite{Kol13}*{4.20}, it suffices to show that the minimal lc center is unique. Suppose that there are two minimal lc centers $Z_1\neq Z_2$ that intersect $f^{-1}(u)$. They are necessarily disjoint, otherwise their intersection contains smaller lc centers. Let $I_{Z_1\cup Z_2}$ be the ideal sheaf of $Z_1\cup Z_2\subseteq X$. Let $m>0$ be a sufficiently large and divisible integer such that $\cO_X(mL)\otimes I_{Z_1\cup Z_2}$ is globally generated, and let $G\in |\cO_X(mL)\otimes I_{Z_1\cup Z_2}|$ be a general member. Fix some $0<\varepsilon\ll 1$. Then $-(K_X+D+\varepsilon G)$ is ample, $(X,D+\varepsilon G)$ is lc away from $Z_1\cup Z_2$ (by Bertini's theorem), but is not lc at the generic point of $Z_1$ and $Z_2$. Consider the convex combination
\[
(X,\Gamma:=c\Delta+(1-c)(D+\varepsilon G))
\]
of $(X,D+\varepsilon G)$ with the klt pair $(X,\Delta)$, where $0<c\ll 1$. Then $(X,\Gamma)$ is klt away from $Z_1\cup Z_2$, its non-klt locus is exactly $Z_1\cup Z_2$, and $-(K_X+\Gamma)$ is ample. Since $Z_1$ is disjoint from $Z_2$, this contradicts the Koll\'ar-Shokurov connectedness theorem \cite{K+92}*{17.4 Theorem}. Finally, if $W$ is an lc center of $(X,D)$ and $W\to U$ has disconnected fibers, then after replacing $U$ with an \'etale local neighbourhood we may assume that $W$ itself is disconnected (see e.g. \cite{Kol13}*{Claim 4.38.1}). But this contradicts the uniqueness of the minimal lc center.
\end{proof}

To prepare for the study of the degenerations of Koll\'ar models in the next section, we also introduce a family version of Koll\'ar models. Let us recall that a log canonical pair $(X,\Delta)$ with a flat morphism $f\colon X\to B$ to a smooth base is called \emph{a locally stable family} (see \cite[Section 4.6]{Kol23}), if for any simple normal crossing divisor $H_1+\dots+H_{r}\subseteq B$ ($r=\dim B$), the pair
$(X,\Delta+f^{*}(H_1+\cdots+H_r))$ is log canonical. 

\begin{defn} \label{defn:km family}
Let $B\subseteq (X,\Delta)\to B$ be a $\bQ$-Gorenstein family of klt singularities over a smooth base $B$ of dimension $r$ (i.e. it is a locally stable family with a section such that $b\in (X_b,\Delta_b)$ is klt for all $b\in B$). A \emph{locally stable family of Koll\'ar models} of $B\subseteq (X,\Delta)\to B$ is a model $\pi\colon (Y,E)\to (X,\Delta)$ such that the following holds:
\begin{enumerate}
    \item $\pi$ is an isomorphism away from $B\subseteq X$ and $E=\Ex(\pi)$,
    \item $-(K_Y+\pi_*^{-1}\Delta+E)$ is $\pi$-ample, and
    \item for any closed point $b\in B$, there exists some SNC divisor $H_1+\dots+H_r$ with $b\in H_1\cap\dots\cap H_r$ such that $(Y,\pi_*^{-1}\Delta+E+p^*(H_1+\dots+H_r))$ is qdlt where $p$ denotes the map $Y\to B$.
\end{enumerate}
When $E$ is irreducible, we call $(Y,E)\to (X,\Delta)\to B$ a {\it locally stable family of Koll\'ar components}.
\end{defn}

\begin{lem} \label{lem:flat limit}
Using notation from Definition \ref{defn:km family}, we have
\begin{enumerate}
    \item Every irreducible component of $E$ dominates $B$.
    \item For any irreducible component $E_1$ of $E$, the map $p|_{E_1}\colon E_1\to B$ is equidimensional with irreducible and reduced fibers.
    \item $\pi_b\colon (Y_b,E_b)\to (X_b,\Delta_b)$ is a Koll\'ar model for all $b\in B$.
    \item The component $E_1$ can be extracted as a locally stable family of Koll\'ar components, and the natural map 
    \[
    \pi_*\cO_Y(-mE_1)\to \pi_*\cO_{Y_b}(-mE_1)
    \]
    is surjective for all $b\in B$ and $m\in\bN$.
\end{enumerate}
\end{lem}

\begin{proof}
Fix some $b\in p(E_1)$ and let $D=\pi_*^{-1}\Delta$. Since all the statements are local on both $X$ and $B$, we can shrink $X$ and $B$ around $b$ if necessary. Let $H=H_1+\dots+H_r$. Since $-(K_Y+D+E)$ is ample over $X$ and $E_1$ is an lc center of $(Y,D+E)$, by Lemma \ref{lem:lc center of Fano pair} we know that the map $p\colon E_1\to B\subseteq X$ has connected fibers. 
On the other hand, since $(Y,D+E+p^*H)$ is qdlt,  we know that the fiber $$E_{1,b}=p^{-1}(b)\cap E_1 = H_1\cap\dots\cap H_r\cap E_1$$ is normal by Lemma \ref{lem:qdlt->normal lc center}, and $$\dim E_{1,b} = \dim Y - (r+1) = \dim X_b - 1.$$ It follows that $E_{1,b}$ is irreducible and reduced, and the map $E_1\to B$ is equidimensional. This gives the statement (2). It also implies that $E_1$ dominates $B$ as otherwise $$\dim E_{1,b}\ge \dim E_1-\dim B+1\ge \dim Y-r,$$ a contradiction. By adjunction, we also see that $\pi_b\colon (Y_b,E_b)\to (X_b,\Delta_b)$ is a Koll\'ar model. Since $E_1$ is arbitrary, we obtain the statements (1) and (3).


By Lemma \ref{lem:perturb qdlt boundary}, we get a $\bQ$-divisor $D_0\ge \pi_*^{-1}\Delta$ such that $-(K_Y+D_0+E_1)$ is $\bQ$-Cartier and $\pi$-ample, $\lfloor D_0\rfloor =0$ and $(Y,D_0+E_1+p^*H)$ is qdlt. Let $D_1$ be the sum of $D_0$ with a general member of $\frac{1}{m}|-m(K_Y+D_0+E_1)|$ for some sufficiently divisible integer $m>0$. Then $D_1\ge \pi_*^{-1}\Delta$, $\lfloor D_1\rfloor =0$, $$K_Y+D_1+E_1\sim_{\bQ,\pi} 0,$$ and $(Y,D_1+E_1+p^*H)$ is qdlt. Then $E_1$ is an lc place of the pair $(X,\Delta+\Gamma)$ where $\Gamma=\pi_*D_1-\Delta\ge 0$. By \cite{BCHM}*{Corollary 1.4.3}, there exists a birational contraction $g\colon Y\dashrightarrow Y'$ over $X$ such that $E'=g_*E_1$ is the unique exceptional divisor of $\pi'\colon Y'\to X$, and $-E'$ is ample over $X$. We claim that
$$(Y',E')\to X\to B$$ extracts $E_1$ as a locally stable family of Koll\'ar components. 

To see this, denote the map $Y'\to B$ by $p'$ and let $D'=g_*D_1$. Since 
$$K_Y+D_1+E_1\sim_{\bQ,\pi} 0$$ and $(Y,D_1+E_1+p^*H)$ is lc, the same holds for its birational transform $(Y',D'+E'+p'^*H)$. Using $D_1\ge \pi_*^{-1}\Delta$ we then see that $$(Y',\pi'^{-1}_*\Delta+E')\to B$$ is locally stable. This implies $E'$ is equidimensional over $B$ with reduced fibers. In particular, it does not contain any component of $Y_b$. As $-E'$ is ample over $X$, the $\pi'$-exceptional locus is supported on $E'$, thus we also deduce that $Y'_b\dashrightarrow X_b$ is birational. By adjunction we know that $(Y'_b,D'_b+E'_b)$ is lc, and it is crepant to $(Y_b,D_{1,b}+E_{1,b})$. Recall that $(Y,D_1+E_1+p^*H)$ is qdlt and $\lfloor D_1\rfloor =0$, thus by adjunction $(Y_b,D_{1,b}+E_{1,b})$ is qdlt and its only lc place is $E_{1,b}$ (recall that $E_{1,b}$ is irreducible). Since every component of $E'_b$ has to be an lc place of $(Y'_b,D'_b+E'_b)$, we see that $E'_b$ is also irreducible and it is the unique lc center of $(Y'_b,D'_b+E'_b)$, which in turn gives the minimal lc center of $(Y',D'+E'+p'^*H)$. Moreover, the birational map $Y_b\dashrightarrow Y'_b$ identifies $E_{1,b}$ with $E'_b$ as divisors over $X_b$. By Lemma \ref{lem:qdlt criterion}, we conclude that the pair $(Y',D'+E'+p'^*H)$ is also qdlt. So $(Y',E')\to B$ is a locally stable family of Koll\'ar model as desired, which gives the first part of the statement (4). 

Finally, by Kawamata-Viehweg Vanishing Theorem and the fact that $-E'$ is $\pi'$-ample we have $R^1\pi'_*\cO_{Y'}(-mE'-p^*H_i)=0$. Repeatedly cutting down by the hyperplanes $H_i$, we have $$\pi'_*\cO_{Y'}(-mE')\to \pi'_*\cO_{Y'_b}(-mE')$$ is surjective. Clearly $\pi'_*\cO_{Y'}(-mE') = \pi_*\cO_Y(-mE_1)$ as $E_1$ and $E'$ are the same divisors over $X$. Similarly $\pi'_*\cO_{Y'_b}(-mE')=\pi_*\cO_{Y_b}(-mE_1)$ as the birational map $Y_b\dashrightarrow Y'_b$ identifies $E_{1,b}$ with $E'_b$ as divisors over $X_b$. Therefore, $\pi_*\cO_Y(-mE_1)\to \pi_*\cO_{Y_b}(-mE_1)$ is also surjective. This proves the remaining part of (4). The proof is now complete.
\end{proof}

\subsection{Construction of Koll\'ar models}\label{ss-constructingmodel}

In the rest of this section, we explain how to construct Koll\'ar models, or more generally, models of qdlt Fano type using special $\bQ$-complements. The main goal is to prove the following statement.

\begin{thm} \label{thm:construction of qdlt fano type model}
Let $(X,\Delta)$ be a klt pair that is projective over a quasi-projective variety $U$ such that $-K_X-\Delta$ is ample. Let $\pi\colon (Y,E)\to (X,\Delta)$ be a toroidal model, and let $\Gamma$ be a special $\bQ$-complement with respect to $\pi$. Then for any rational simplex $\sigma\subseteq\LC(\Gamma;Y,E)$, there exists a model $\pi'\colon (Y',E')\to (X,\Delta)$ of qdlt Fano type such that $\QM(Y',E')=\sigma$ as a subset of $\Val_X$. 
\end{thm}

Here we say $\sigma\subseteq\LC(\Gamma;Y,E)$ is a \emph{rational simplex} if there exist divisorial valuations $v_1,\dots,v_r\in \LC(\Gamma;Y,E)$ that are contained in the same simplicial cone of $\QM(Y,E)$ such that they span an $(r-1)$-simplex and $\sigma$ is the cone over it.

We divide the proof of this theorem in two major steps: a tie-breaking step, where we modify the toroidal model $(Y,E)$ and the special $\bQ$-complement $\Gamma$ while fixing the rational simplex $\sigma$ so that $$\QM(Y,E)=\LC(X,\Delta+\Gamma)=\sigma;$$ and an MMP step, where we contract all the exceptional divisors that are not contained in $E$ to obtain the qdlt Fano type model. We first treat the MMP step.

\begin{lem} \label{lem:MMP step}
Theorem \ref{thm:construction of qdlt fano type model} holds when $\sigma=\QM(Y,E)$.
\end{lem}

\begin{proof}
Since in general $\sigma\subseteq \LC(\Gamma;Y,E) \subseteq\QM(Y,E)$, the assumption in particular includes that $\QM(Y,E)$ is a simplex and $$\QM(Y,E)=\LC(\Gamma;Y,E).$$ Let $(Y,\Delta_Y)$ be the crepant pullback of $(X,\Delta)$. Since $\Gamma$ is a special $\bQ$-complement, there exists some effective ample $\bQ$-divisor $G\le \pi^{-1}_*\Gamma$ that does not contain any stratum of $E$. Then the assumption $\QM(Y,E)=\LC(\Gamma;Y,E)$ implies that $(Y,\Delta_Y+\pi^*\Gamma)$ is sub-qdlt. 

We first show that it suffices to find a birational contraction $g\colon Y\dashrightarrow Y'$ with $Y'$ being projective over $X$, such that
\begin{enumerate}
    \item $g$ is an isomorphism at the generic point of every stratum of $(Y,E)$, and
    \item $g$ contracts all the $\pi$-exceptional divisors that are not contained in $E$.
\end{enumerate}
These conditions are natural in order for $(Y',E'=g_*E)$ to be a model of qdlt Fano type: in fact condition (1) is the same as the assumption in Proposition \ref{prop:modify qdlt Fano}(2), while condition (2) guarantees that $Y'$ is of Fano type since $g$ in particular contracts all the $\pi$-exceptional divisors with non-positive coefficients in $\Delta_Y+\pi^*\Gamma$ (these are the divisors that potentially prevent $Y$ from being of Fano type). To see that the two conditions are sufficient for us, let $G'_0$ be an effective ample $\mathbb{Q}$-divisor on $Y'$ that is in a general position, and let $G_0$ be its birational transform on $Y$. Then $\Supp(G_0)$ does not contain any stratum of $E$ (by (1) and Bertini), which are also the lc centers of $(Y,\Delta_Y+\pi^*\Gamma)$. It follows that for $0<\varepsilon\ll 1$ ($\varepsilon\in\bQ$) we have $G-\varepsilon G_0$ is still ample, and $(Y,\Delta_Y+\pi^*\Gamma-G+\varepsilon G_0)$ is still sub-qdlt with
\[
\LC(Y,\Delta_Y+\pi^*\Gamma-G+\varepsilon G_0)=\QM(Y,E).
\]
Choose a sufficiently divisible integer $m$ and take a general $G_1\in \frac{1}{m}|m(G-\varepsilon G_0)|$. Then by Bertini's theorem $(Y,\Delta_Y+\pi^*\Gamma-G+\varepsilon G_0+G_1)$ is also sub-qdlt with
\[
\LC(Y,\Delta_Y+\pi^*\Gamma-G+\varepsilon G_0+G_1)=\LC(Y,\Delta_Y+\pi^*\Gamma-G+\varepsilon G_0)=\QM(Y,E).
\]
Moreover, as
$$\pi^*\Gamma-G+\varepsilon G_0+G_1\sim_{\bQ}\pi^*\Gamma$$ and the left hand side is an effective $\bQ$-divisor we have 
\begin{equation} \label{eq:Gamma'>=ample}
    \pi^*\Gamma-G+\varepsilon G_0+G_1 = \pi^*\Gamma'
\end{equation}
for some $\bQ$-complement $\Gamma'$ of $(X,\Delta)$. Note that $K_Y+\Delta_Y+\pi^*\Gamma'\sim_{\bQ,\pi} 0$. By construction, the lc places of $(X,\Delta+\Gamma')$ are still given by $\QM(Y,E)$. Denote the induced map $Y'\to X$ by $\pi'$. By the property (2) of the birational contraction $g\colon Y\dashrightarrow Y'$, the birational transform $g_*(\Delta_Y+\pi^*\Gamma')$ is effective; in other words, the crepant pullback of $(X,\Delta+\Gamma')$ to $Y'$ has an effective boundary. Combined with the property (1), we see that $(Y',g_*(\Delta_Y+\pi^*\Gamma'))$ is qdlt and its lc places are given by $\QM(Y',E')$. By \eqref{eq:Gamma'>=ample}, ${\pi_*'}^{-1}\Gamma'\ge \varepsilon G_0$. Let 
$$D=g_*(\Delta_Y+\pi^*\Gamma)-\varepsilon G_0\ge {\pi_*'}^{-1}\Delta.$$ 
Then $(Y',D)$ is qdlt, $\lfloor D\rfloor =E'$ and $-(K_{Y'}+D)\sim_\bQ \varepsilon G_0$ is ample. It follows that the model $(Y',E')\to (X,\Delta)$ is qdlt of Fano type as desired.

Thus it remains to find a birational contraction that satisfies (1)--(2). Let $F$ be the sum of $\pi$-exceptional divisors that are not contained in $E$. We aim to run an MMP $Y\dashrightarrow Y'$ over $X$ for some effective divisor whose support is $F$. Since $F$ does not contain any stratum of $E$, the MMP process satisfies property (1). By the negativity lemma \cite{KM98}*{Lemma 3.39}, the MMP will contract $F$ hence it also satisfies property (2). Nevertheless, we need to verify that the MMP exists since $Y$ may not be of Fano type over $X$. To this end, fix some $0<c\ll 1$ so that $A_{X,\Delta+(1-c)\Gamma}(E_i)<1$ for every irreducible component $E_i$ of $E$ (this is possible since $E_i$ is an lc place of $(X,\Delta+\Gamma)$) and write
\[
\pi^*(K_X+\Delta+(1-c)\Gamma)=K_Y+\Delta_1-\Delta_2
\]
where $\Delta_1$ and $\Delta_2$ are effective and have no common components. Then $\Supp(\Delta_2)\subseteq F$. As $(X,\Delta+(1-c)\Gamma)$ is klt, we also know that the coefficients of $\Delta_1$ are less than $1$. 

If $(Y,\Delta_1)$ is klt, then so does $(Y,\Delta_1+\varepsilon F)$ when $0<\varepsilon\ll 1$, and we may run the $(K_Y+\Delta_1+\varepsilon F)$-MMP $Y\dashrightarrow Y'$ over $X$ by \cite{BCHM}*{Theorem 1.2}. Since
\[
K_Y+\Delta_1+\varepsilon F\sim_{\bQ,\pi} \Delta_2+\varepsilon F
\]
and the right hand side is supported on $F$, this gives the desired birational contraction.

In general, since $Y$ is $\mathbb{Q}$-factorial, we can take a dlt modification by \cite{OX12}*{Theorem 1.1} of $(Y,\Delta_1)$. This is a projective birational morphism $\rho\colon Z\to Y$ with exceptional divisors $F'_i$ such that $Z$ is $\bQ$-factorial, $(Z,\rho_*^{-1}\Delta_1+\sum F'_i)$ is dlt, and $K_Z+\rho_*^{-1}\Delta_1+\sum F'_i$ is $\rho$-nef. By the negativity lemma, the last condition implies that $A_{Y,\Delta_1}(F'_i)\le 0$ for all $F'_i$. In particular, $\rho$ is an isomorphism over the locus where $(Y,\Delta_1)$ is already klt. We claim that $\rho$ is an isomorphism at the generic point of every stratum of $E$, so that we can try to construct the birational contraction by running MMP on $Z$ instead.

For this it suffices to show that if $\eta\in Y$ is the generic point of any stratum of $E$, then $(Y,\Delta_1)$ is klt at $\eta$. But since $(Y,E+F)$ is toroidal and $\Supp(\Delta_2)\subseteq F$, we have $\eta\not\in\Supp(\Delta_2)$. Thus as the crepant pullback $(Y,\Delta_1-\Delta_2)$  of the klt pair $(X,\Delta+(1-c)\Gamma)$ is sub-klt, we deduce that $(Y,\Delta_1)$ is klt at $\eta$ as desired, and therefore $\rho\colon Z\to Y$ is an isomorphism at $\eta$.

Thus we may replace $Y$ by $Z$ and run through the above argument. Let $\varphi\colon Z\to X$ be the induce map, and write
\[
\varphi^*(K_X+\Delta+(1-c)\Gamma)=K_Z+D_1-D_2
\]
where $D_1$ and $D_2$ are effective and have no common components. Let $\tF$ be the sum of all $\varphi$-exceptional divisors that are not contained in $\rho_*^{-1}E$ (i.e. $\tF=\rho_*^{-1}F+\sum F'_i$). As before, we have $\Supp(D_2)\subseteq \tF$. We claim that $(Z,D_1)$ is klt. Clearly $\lfloor D_1 \rfloor=0$ since $(X,\Delta+(1-c)\Gamma)$ is klt; thus $$D_1\le \rho_*^{-1}\Delta_1+(1-\varepsilon)\sum F'_i$$ for some $0<\varepsilon\ll 1$. Since $\lfloor \Delta_1 \rfloor=0$ and $(Z,\rho_*^{-1}\Delta_1+\sum F'_i)$ is dlt, we know that $(Z,\rho_*^{-1}\Delta_1+(1-\varepsilon)\sum F'_i)$ is klt, hence so is $(Z,D_1)$, proving the above claim.

Now we can complete the proof. Since $(Z,D_1+\varepsilon \tF)$ remains klt when $0<\varepsilon\ll 1$, we may run the $(K_Z+D_1+\varepsilon \tF)$-MMP $Z\dashrightarrow Y'$ over $X$. As 
\[
K_Z+D_1+\varepsilon \tF\sim_{\bQ,\varphi} D_2+\varepsilon \tF
\]
and the right hand side is fully supported on $\tF$, the MMP exactly contracts $\tF$ and thus the induced birational map $Y\dashrightarrow Y'$ is a birational contraction that satisfies the property (2). As $\tF$ does not contain any stratum of $\rho_*^{-1}E$ and $\rho\colon Z\to Y$ is an isomorphism at the generic point of every stratum of $E$, the map $Y\dashrightarrow Y'$ also satisfies the property (1). This finishes the proof.
\end{proof}

We next work towards reducing Theorem \ref{thm:construction of qdlt fano type model} to the special case treated above. For this we need an auxiliary result.

\begin{lem} \label{lem:ample model of a simplex}
Let $(Y,E)$ be a toroidal pair, and let $\Delta$ be a $\bQ$-divisor supported on $E$ such that $\lfloor \Delta\rfloor \le 0$. Let $F_1,\dots,F_r$ be toroidal divisors over $(Y,E)$ such that $v_i=\ord_{F_i}$ are in a simplicial cone of $\QM(Y,E)$ and span an $(r-1)$-simplex. Then there exists a proper birational morphism $\rho\colon Z\to Y$ extracting the divisors $F_1,\dots,F_r$ such that $\Ex(\rho)\subseteq \cup_{i=1}^r F_i$ and $-\sum_{i=1}^r A_{Y,\Delta}(F_i)\cdot F_i$ is ample over $Y$.
\end{lem}

\begin{proof}
Let $f\colon W\to Y$ be a toroidal blowup that extracts the divisors $F_1,\dots,F_r$. Then $W$ is of Fano type over $Y$. By running an MMP $g\colon W\dashrightarrow W'$ with scaling over $Y$ we obtain a model such that $-\sum_{i=1}^r A_{Y,\Delta}(F_i)\cdot g_* F_i$ is nef over $Y$ and let $h\colon W'\to Z$ be the corresponding ample model over $Y$. It suffices to show that none of the divisors $F_i$ are contracted in this process. By assumption, $D:=E-\Delta$ is effective and $\Supp(D)=E$. Since $A_{Y,E}(F_i)=0$, we have $A_{Y,\Delta}(F_i)=\ord_{F_i}(D)$, thus 
\[
-\sum_{i=1}^r A_{Y,\Delta}(F_i)\cdot F_i\sim_{\bQ,f} f^*D-\sum_{i=1}^r \ord_{F_i}(D)\cdot F_i = D_W,
\]
for some effective divisor $D_W$ that does not contain any $F_i$ in its support. It follows that the ($D_W$-)MMP $W\dashrightarrow W'$ does not contract any of the divisors $F_i$ and hence by replacing the initial model $W$ with $W'$ we may simply assume $W=W'$. On the other hand, we have
\[
K_W+\Delta_W+\sum_{i=1}^r F_i\sim_{\bQ,f} \sum_{i=1}^r A_{Y,\Delta}(F_i)\cdot F_i,
\]
thus the ample model $W\to Z$ satisfies
\[
K_W+\Delta_W+\sum_{i=1}^r F_i = h^*(K_Z+\Delta_Z+\sum_{i=1}^r h_* F_i).
\]
Here $\Delta_Z,\Delta_W$ denote the strict transform of $\Delta$ on $Z,W$. Note that locally both $(W,\Delta_W+\sum_{i=1}^r F_i)$ and $(Z,\Delta_Z+\sum_{i=1}^r h_*F_i)$ have toroidal singularities. In fact,  \'etale locally on $Y$ the toroidal pair $(Y,E)$ is isomorphic to a toric pair, and if we take a base change of our MMP over the  toric pair, the MMP with scaling and the ample model construction are naturally torus equivariant. Also recall that $\lfloor \Delta\rfloor \le 0$. Thus $\LC(W,\Delta_W+\sum_{i=1}^r F_i)$ is the $r$-dimensional simplicial cone spanned by all $F_i$ while $\LC(Z,\Delta_Z+\sum_{i=1}^r h_*F_i)$ is the simplicial cone spanned by $h_* F_i$. But as the two pairs are crepant equivalent, their dual complexes have the same dimension. In particular, the divisors $h_* F_i$ also span an $(r-1)$-simplex. This implies that none of the $F_i$'s are contracted on the ample model.
\end{proof}

The following lemma gives the tie-breaking step for the proof of Theorem \ref{thm:construction of qdlt fano type model}.

\begin{lem} \label{lem:DMR=specified QM}
Let $\pi\colon (Y,E)\to (X,\Delta)$ be a toroidal model and let $\Gamma_Y$ be a special $\bQ$-complement of $(X,\Delta)$ with respect to $\pi$. Let $F_1,\dots,F_r$ be divisors over $Y$ such that $v_i=\ord_{F_i}\in \LC(\Gamma_Y;Y,E)$ are in a simplicial cone of $\QM(Y,E)$ and they span an $(r-1)$-simplex. Then there exists a birational toroidal morphism $\rho\colon Z\to (Y,E)$ and a $\bQ$-complement $\Gamma$ of $(X,\Delta)$ such that
\begin{enumerate}
    \item $\rho$ extracts the divisors $F_1,\dots,F_r$ and $\Ex(\rho)\subseteq \cup_{i=1}^r F_i$;
    \item $\Gamma$ is special with respect to $(Z,F)$ where $F=F_1+\dots+F_r$; and
    \item $\QM(Z,F)=\LC(X,\Delta+\Gamma)$.
\end{enumerate}
\end{lem}

\begin{proof}
Let $K_Y+\Delta_Y=\pi^*(K_X+\Delta)$ be the crepant pullback. Then $(Y,\Delta_Y)$ is sub-klt and in particular $\lfloor \Delta_Y \rfloor \le 0$. By applying Lemma \ref{lem:ample model of a simplex} to the toroidal pair $(Y,\Supp(E+\pi_*^{-1}\Delta+\Ex(\pi)))$ and the sub-boundary $\Delta_Y$, we deduce that there exists a toroidal birational morphism $\rho\colon Z\to Y$ extracting the divisors $F_1,\dots,F_r$ such that $-\sum_{i=1}^r A_{Y,\Delta_Y}(F_i)\cdot F_i$ is ample over $Y$. Note that $A_{Y,\Delta_Y}(F_i)=A_{X,\Delta}(F_i)$, so this $\bQ$-divisor is the same as $-\sum_{i=1}^r A_{X,\Delta}(F_i)\cdot F_i$. To prove the lemma, we need to find a special $\bQ$-complement $\Gamma$ with respect to $(Z,F)$ such that $\QM(Z,F)=\LC(X,\Delta+\Gamma)$.

First we need to introduce some more notation. Let $(Z,\Delta_Z)$ be the crepant pullback of $(X,\Delta)$, i.e.,
$$\rho^* \pi^*(K_X+\Delta)=K_Z+\Delta_Z,$$
and let $\tF=\sum_{i=1}^r A_{X,\Delta}(F_i)\cdot F_i$. Since $\Gamma_Y$ is a special $\bQ$-complement with respect to $(Y,E)$, we have $\pi^{-1}_* \Gamma_Y\ge G$ for some effective ample $\bQ$-divisor $G$ that does not contain any stratum of $(Y,E)$. Let $$D=\pi^*\Gamma_Y-G\ge 0.$$ 
Since $G$ is ample on $Y$ and $-\tF$ is ample over $Y$, we can choose a rational number $0<\varepsilon\ll 1$ such that both $\frac{1}{2}G+\varepsilon D$ and $\frac{1}{2}\rho^*G-\varepsilon \tF$ are ample. This guarantees that $\rho^*(G+\varepsilon D)-\varepsilon \tF$ is ample.

We break the pullback of $\Gamma_Y$ to $Z$ into three parts:
\[
\rho^*\pi^* \Gamma_Y = (1-\varepsilon)\rho^*D + \varepsilon \tF + (\rho^*(G+\varepsilon D)-\varepsilon \tF).
\]
We claim that
\begin{equation} \label{eq:DMR=simplex}
    \LC(Z,\Delta_Z+(1-\varepsilon)\rho^*D + \varepsilon \tF) = \QM(Z,F).
\end{equation}
Once this is proved, the lemma follows almost immediately: since $\rho^*(G+\varepsilon D)-\varepsilon \tF$ is ample, by Bertini's theorem we can choose some effective $\bQ$-divisor 
$$G'\sim_\bQ \rho^*(G+\varepsilon D)-\varepsilon \tF$$ such that
\[
\LC(Z,\Delta_Z+(1-\varepsilon)\rho^*D + \varepsilon \tF+G') = \QM(Z,F)
\]
still holds. In particular, the ample $\bQ$-divisor $G'$ does not contain any stratum of $(Z,F)$ in its support. Since 
$$(1-\varepsilon)\rho^*D + \varepsilon \tF+G'\sim_\bQ \rho^*\pi^*\Gamma_Y\sim_{\bQ} 0\,/X,$$ we have 
\[
(1-\varepsilon)\rho^*D + \varepsilon \tF+G'=\rho^*\pi^*\Gamma
\]
for the effective $\bQ$-divisor 
$$\Gamma:= \pi_*\rho_*((1-\varepsilon)\rho^*D + \varepsilon \tF+G')\sim_\bQ \Gamma_Y$$ on $X$. By construction $\Gamma$ is a special $\bQ$-complement with respect to $(Z,F)$ and $$\rho^*\pi^*(K_X+\Delta+\Gamma)=K_Z+\Delta_Z+(1-\varepsilon)\rho^*D + \varepsilon \tF+G'.$$ Thus $$\LC(X,\Delta+\Gamma)=\LC(Z,\Delta_Z+(1-\varepsilon)\rho^*D + \varepsilon \tF+G') = \QM(Z,F)$$ as desired.

It remains to prove \eqref{eq:DMR=simplex}. The pair in question is a convex linear combination of $(Z,\Delta_Z+\rho^*D)$ and $(Z,\Delta_Z+\tF)$, thus it suffices to show that:
\begin{enumerate}
    \item both $(Z,\Delta_Z+\rho^*D)$ and $(Z,\Delta_Z+\tF)$ are sub-lc,
    \item $\QM(Z,F)\subseteq \LC(Z,\Delta_Z+\rho^*D)$, and
    \item $\LC(Z,\Delta_Z+\tF)=\QM(Z,F)$.
\end{enumerate}
First, $(Z,\Delta_Z+\rho^*D)$ is sub-lc since it is the crepant pullback of $(Y,\Delta_Y+D)$, which is sub-lc as 
$$K_Y+\Delta_Y+D\le \pi^*(K_X+\Delta+\Gamma_Y).$$ Moreover, all the $F_i$'s are lc places of $(Y,\Delta_Y+D)$ by assumption, thus 
$$\QM(Z,F)\subseteq \LC(Z,\Delta_Z+\rho^*D).$$ On the other hand, a direct computation gives
\[
K_Z+\Delta_Z+\tF = \rho^*(K_Y+\Delta_Y)+\sum_{i=1}^r A_{Y,\Delta_Y}(F_i)\cdot F_i = K_Z+\left(\rho^{-1}_*\Delta_Y \vee \sum_{i=1}^r F_i\right),
\]
where the notation $D_1\vee D_2$ denotes the smallest $\bQ$-divisor $D$ such that $D\ge D_i$ for $i=1,2$. As $\lfloor \Delta_Y \rfloor \le 0$, we hence see that the toroidal pair $(Z,\Delta_Z+\tF)$ is also lc and its lc places are exactly given by $\QM(Z,F)$. Thus we have proved all the properties (1)--(3) above and this finishes the proof.
\end{proof}

\begin{proof}[Proof of Theorem \ref{thm:construction of qdlt fano type model}]
By Lemma \ref{lem:DMR=specified QM}, after possibly replacing the toroidal model $(Y,E)$ and the corresponding special $\bQ$-complement $\Gamma$, we may assume that $\sigma=\QM(Y,E)$. The result then follows from Lemma \ref{lem:MMP step}.
\end{proof}

One of the main applications of Theorem \ref{thm:construction of qdlt fano type model} is the construction of Koll\'ar models for monomial lc places of special $\bQ$-complements.

\begin{cor} \label{cor:km for monomial lc place}
Let $x\in (X,\Delta)$ be a klt singularity and let $v\in \Val_{X,x}$ be a monomial lc place of some special $\bQ$-complement $\Gamma$ with respect to some toroidal model $(Y,E)\to (X,\Delta)$. Then there exists a Koll\'ar model $\pi'\colon (Y',E')\to (X,\Delta)$ such that $v\in \QM(Y',E')$.
\end{cor}

\begin{proof}
Let $\Sigma\subseteq \QM(Y,E)$ be the smallest simplicial face that contains $v$ and choose some rational simplex $$\sigma\subseteq {\rm int}(\Sigma)\cap \LC(\Gamma;Y,E)$$ such that $v\in \sigma$. By Theorem \ref{thm:construction of qdlt fano type model}, we get a model $\pi\colon (\tY,\tE)\to (X,\Delta)$ of qdlt Fano type such that $\QM(\tY,\tE)=\sigma$. Since $v$ is centered at $x$, the same holds for all valuations in ${\rm int}(\Sigma)$. In particular $\pi(\tE)=\{x\}$ and thus the ample model $\pi\colon (Y',E')\to (X,\Delta)$ of $-(K_{\tY}+\pi_*^{-1}\Delta+\tE)$ is a Koll\'ar model by Proposition \ref{prop:qdlt to km}. By construction $v\in \sigma=\QM(\tY,\tE)=\QM(Y',E')$.
\end{proof}

\section{Koll\'ar models and finite generation}

The goal of this section is to prove the following finite generation statement, which is a more precise version of Theorem \ref{thm-monomialfinitegeneration}.

\begin{thm} \label{t-localHRFG}
Let $x\in (X=\Spec(R),\Delta)$ be a klt singularity, and let $v\in \Val_{X,x}$ be a quasi-monomial valuation. Then the following are equivalent.
\begin{enumerate}
    \item The associated graded ring $\gr_v R$ is finitely generated and the central fiber $(X_v,\Delta_v)$ of the induced degeneration is klt.
    \item The valuation $v$ is a monomial lc place of a special $\bQ$-complement $\Gamma$ with respect to some toroidal model $(Y,E)$ $($see Definition \ref{d-specialcomplement}$)$. 
    \item There exists a Koll\'ar model $($see Definition \ref{defn:qdlt Fano model}$)$ $\pi\colon (Y,E)\to (X,\Delta)$ such that $v\in \QM(Y,E)$.
\end{enumerate}
\end{thm}

We explain some notation in the above theorem (c.f. \cite{LXZ-HRFG}*{Theorem 4.2}). Assuming the finite generation of $\gr_v R$, then similar to the Rees construction \eqref{eq-rees} for a primitive divisor, we define $X_v:=\Spec(\gr_v R)$, and let $\Delta_v$ be the induced degeneration of $\Delta$ to $X_v$. More precisely, suppose $\Delta=\sum_{i=1}^{l} a_i \Delta_i$ where $\Delta_i$ is a prime divisor on $X$ and $a_i\in \bQ_{\geq 0}$. Let $I_{\Delta_i}\subseteq R$ be the ideal of $\Delta_i$. Let $\mathrm{in}(I_{\Delta_i})\subseteq \gr_v R$ be the initial ideal of $I_{\Delta_i}$, i.e.
$$\mathrm{in}(I_{\Delta_i})=\bigoplus_{\lambda} I_{\Delta_i, \lambda}, \mbox{ where } I_{\Delta_i, \lambda}=\{{\rm Im} (f)\in \fa_{\lambda}(v)/\fa_{>\lambda}(v)\ | \ f\in I_{\Delta_i} \mbox{ and }v(f)\ge \lambda \}.$$ 
Then $\Delta_v:=\sum_{i=1}^l a_i \Delta_{v,i},$ where $\Delta_{v,i}$ is the divisorial part of the closed subscheme $V(\mathrm{in}(I_{\Delta_i}))$ in  $X_v$, i.e. $\Delta_{v,i}$ and $V(\mathrm{in}(I_{\Delta_i}))$ coincide away from a codimension $2$ subset of $X_v$. When $v=\ord_E$ is a divisorial valuation induced by some primitive divisor $E$ over $x\in X$, this coincides with the construction in \eqref{eq-rees}.

Note that one of the implications in Theorem \ref{t-localHRFG}, namely $(2)\Rightarrow (3)$, has already been established in Corollary \ref{cor:km for monomial lc place}.

\subsection{Special degenerations come from lc places of special complements}

As in \cite{LXZ-HRFG}*{Theorem 4.2}, we first prove the easier part of the remaining implications, i.e. $(1)\Rightarrow (2)$, in Theorem \ref{t-localHRFG}. We will need an auxiliary lemma which is a local analog of \cite{BLX-openness}*{Theorem A.2}.

\begin{lem} \label{lem:N-complement}
Let $E$ be a primitive divisor over a klt singularity $x\in (X,\Delta)$.
Assume that the special fiber $x_0\in (X_0,\Delta_0)$ of the Rees construction \eqref{eq-rees} is lc. Then $E$ is an lc place of some $N$-complement of $x\in (X,\Delta)$, where the integer $N$ only depends on the dimension of $X$ and the coefficients of $\Delta$. 
\end{lem}

\begin{proof}
We follow the proof of \cite{BLX-openness}*{Propositions A.3 and A.4}. To begin with, we show that the associated prime blow up $\pi\colon Y\to X$ of $E$ specializes to some $Y_0\to X_0$. Indeed, let $\fa_j=\pi_*\mathcal{O}_Y(-jE)\subseteq \cO_X$ and let $(\cX,\Delta_{\cX})\to \bA^1$ be the $\bG_m$-equivariant family from \eqref{eq-rees}. Then the flat extension of $\fa_j[t,t^{-1}]\subseteq \cO_{X\times (\bA^1\setminus\{0\})}$ to $\cX$ is given by $\widetilde{\fa}_j:=\bigoplus_{i\in\bZ} t^{-i}(\fa_i\cap \fa_j)$. Since 
\[
\bigoplus_{j\in \bN} \widetilde{\fa}_j/\widetilde{\fa}_{j+1}\cong
\bigoplus_{j\in \bN} t^{-j}(\fa_j/\fa_{j+1}) 
\]
is finitely generated, and if we denote by
\[
\cE=\Proj \left(\bigoplus_{j\in\bN}\widetilde{\fa}_j / \widetilde{\fa}_{j+1} \right)\,,
\]
then we have $\cE\cong E\times \bA^1$. This implies $\bigoplus_{j\in\bN}\widetilde{\fa}_j$ is finitely generated. 
Consider the morphism 
\[
\rho\colon \cY:=\Proj_{\cX} \left(\bigoplus_{j\in\bN}\widetilde{\fa}_j \right)\to \cX
\]
over $\bA^1$. Then
\[
\cY\times_{\bA^1} (\bA^1\setminus\{0\})\cong Y\times (\bA^1\setminus\{0\})
\]
and $\rho$ restricts to $\pi\times \mathrm{id}$ over $\bA^1\setminus\{0\}$. The exceptional divisor of $\rho$ is given by $\cE$.
 Let $\Delta_{\cY}=\rho_*^{-1}\Delta_{\cX}$ be the strict transform of $\Delta_{\cX}$, and denote by $Y_0$ the central fiber of $\cY\to \bA^1$.

Let $\ell>0$ be a sufficiently divisible integer such that $\ell E$ is Cartier and $R^1\pi_*\cO_Y(-p\ell E)=0$ for all $p\ge 1$. Through the long exact sequence 
\[
0\to \cO_Y(-(p\ell+1)E)\to \cO_Y(-p\ell E)\to \cO_E(-p\ell E)\to 0,
\]
we get $\fa_{p\ell}/\fa_{p\ell+1}\cong H^0(E,\cO_E(-p\ell E))$. Let $\sigma\colon \cY\to \cY'$ be the quotient by $\mu_\ell$, where the action is via $\mu_{\ell}\subseteq \bG_m$. Let $\cX',\cE',Y'_0,X'_0$ be the induced quotients of $\cX,\cE,Y_0,X_0$, where the corresponding quotient maps are still denoted by $\sigma$. Since the $\mu_\ell$-action is free outside $X_0$ and $Y_0$, we have
\[
K_{\cX}+\Delta_{\cX}+X_0=\sigma^*(K_{\cX'}+\Delta'_{\cX}+X'_0)
\]
for an effective $\bQ$-divisor $\Delta'_{\cX}$ on $\cX'$ and
\[
K_{\cY}+\cE+Y_0+\Delta_{\cY}=\sigma^*(K_{\cY'}+\cE'+Y'_0+\cD')
\]
for an effective $\bQ$-divisor $\cD'$ on $\cY'$. By construction, we have 
\[
X'_0=\Spec\left(\bigoplus_{p\in\bN} \fa_{p\ell}/\fa_{p\ell+1}\right)=\Spec\left(\bigoplus_{p\in\bN} H^0(E,\cO_E(-p\ell E))\right)
\]
is the cone over $E$ with polarization $-\ell E|_E$. As $Y_0\to X_0$ is given by 
\[
{\rm Proj}_{X_0}(\bigoplus_{j\in \bN} \fa_{j}/\fa_{j+1})\to X_0 \,,
\]
$Y_0'={\rm Proj}_{X'_0}(\bigoplus_{p\in \bN} \fa_{p\ell}/\fa_{p\ell+1})$, i.e., $Y'_0\to X'_0$ is the blowup of the cone vertex. Let $E'_0,D'_0$ be the restrictions of $\cE',\cD'$ to $Y'_0$. Note that $E'_0\cong E$, hence we may view $X'_0=X_0/\mu_\ell$ as the cone over $E'_0$.

Denote by $\Delta'_0$ the restriction of $\Delta'_{\cX}$, so it satisfies $\sigma_{|X_0}^*(K_{X'_0}+\Delta'_{0})=K_{X_0}+\Delta_0$. Since the birational transform of $\Delta_{\cX}$ on $\cY$ is $\Delta_{\cY}$, we know the birational transform of $\Delta'_{\cX}$ is $\cD'$. Restricting over $0$, the birational transform of $\Delta'_0$ is $D_0'$, i.e. $\Delta'_0$ is the cone over $D'_0|_{E'_0}$. By assumption and \cite{KM98}*{Proposition 5.20}, $(X_0,\Delta_0)/\mu_\ell$ is lc, hence by \cite{Kol13}*{Section 3.1} we know that $(Y'_0,E'_0+D'_0)$ is also lc. By inversion of adjunction \cite{Kol13}*{Theorem 4.9}, this implies that $(\cY',Y'_0+\cE'+\cD')$ and its crepant cover $(\cY,Y_0+\cE+\Delta_{\cY})$ are both log canonical (a priori we only know they are lc around $Y'_0$, but using the $\bG_m$-action we deduce that the lc condition propagates). In particular, by restricting to the general fiber we see that $(Y,E+\pi_*^{-1}\Delta)$ is lc. Since $(X,\Delta)$ is klt and $-E$ is $\pi$-ample, we also have
\[
-(K_Y+E+\pi_*^{-1}\Delta)\sim_{\bQ,\pi} -A_{X,\Delta}(E)\cdot E
\]
is $\pi$-ample. Hence by the boundedness of relative complements \cite{Birkar-boundedcomplements}*{Theorem 1.8}, there exists an $N$-complement $\Gamma$ of $(Y,E+\pi_*^{-1}\Delta)$ over $x\in X$, where the integer $N$ only depends on the dimension of $X$ and the coefficients of $\Delta$. Then $\pi_*\Gamma$ is an $N$-complement of $x\in (X,\Delta)$ having $E$ as an lc place.
\end{proof}

We now prove the $(1)\Rightarrow (2)$ direction in Theorem \ref{t-localHRFG}.

\begin{lem} \label{lem:fg->special complement}
Let $x\in (X=\Spec(R),\Delta)$ be a klt singularity, and let $v\in \Val_{X,x}$ be a quasi-monomial valuation. Assume that $\gr_v R$ is finitely generated and $(X_v,\Delta_v)$ is klt. Then there exists a log smooth model $(Y,E)$ of $(X,\Delta)$ such that $v$ is a monomial lc place of a special $\bQ$-complement with respect to $(Y,E)$.
\end{lem}

\begin{proof}
The argument is the same as the ones in \cite{LXZ-HRFG}*{Lemma 4.3}. Let $q$ be the rational rank of $v$. Let $\pi\colon (Y,E)\to (X,\Delta)$ be a log smooth model such that $-F$ is $\pi$-ample for some $\pi$-exceptional effective divisor $F$ and that $v\in \QM_\eta(Y,E)$ for some codimension $q$ point $\eta\in Y$. Let $m_0\in\bN$ be a sufficiently large and divisible integer, let $f\in \pi_*\cO_Y(-m_0 F)$ be a general element, and let $D={\rm div}(f)\subseteq X$. Then the strict transform $G:=\pi_*^{-1}D\sim \pi^*D-m_0 F$ is ample and does not contain any stratum of $(Y,E)$. Since $\gr_v R$ is finitely generated, we can choose some $f_0:=f,f_1,\dots,f_\ell\in R$ whose restrictions form a (finite) set of generators $\bar{f}_0,\dots, \bar{f}_\ell$ of $\gr_v R$ (in particular, $f_0,\dots, f_\ell$ generates $R$). By enlarging the set of generators, we may also assume that all $\mathrm{in}(I_{\Delta_i})\subseteq \gr_v R$ are generated by the restrictions of some elements from $f_0,\dots,f_{\ell}$.

Choose some rational number $0<\varepsilon\ll 1$ such that $(X_v,\Delta_v+\varepsilon D_v)$ remains klt. By \cite{LX18}*{Lemma 2.10}, there exists a small neighbourhood $U\subseteq \QM_\eta(Y,E)$ of $v$ such that for all valuations $w\in U$, we have an isomorphism $\gr_w R\cong \gr_v R$ sending the restrictions of $f_0,\dots,f_\ell$ in $\gr_v R$ to their respective restrictions in $\gr_w R$. Then $$(X_w,\Delta_w+\varepsilon D_w)\cong (X_v,\Delta_v+\varepsilon D_v)$$ is also klt for any divisorial valuations $w\in U$. By Lemma \ref{lem:fg=primitive}, we know that $w$ is induced by some primitive divisor over $x\in X$, and thus by Lemma \ref{lem:N-complement} we further deduce that there exists an integer $N$ that only depends on $\dim(X)$ and the coefficients of $(X,\Delta+\varepsilon D)$, such that for any divisorial valuation $w_0\in U$, there is an $N$-complement $\Gamma_0$ of $(X,\Delta+\varepsilon D)$ such that $w_0$ is an lc place of $(X,\Delta+\varepsilon D+\Gamma_0)$. By enlarging $N$, we may assume that $N\Gamma_0\sim -N(K_X+\Delta+\varepsilon D)$ is Cartier.

Recall that $v(g)$ ($g\in R$) is computed as the smallest weight of monomials in the power series expansion of $g$ at the point $\eta$. As $\Gamma$ varies among the $N$-complements, we have $\mult_\eta(\pi^*\Gamma)\le C$ for some absolute constant $C>0$, as otherwise $(X,\Delta+\Gamma)$ would not be lc. Possibly after shrinking $U$, there also exist some constants $a,b>0$ depending only on $v$ such that $$a\cdot \mult_\eta(u) \le w(u)\le b\cdot\mult_\eta(u)$$ for all monomials $u\in\hat{\cO}_{Y,\eta}$ and all $w\in U$. In particular, all the monomials that are responsible for computing $v(\Gamma)$ have bounded multiplicity (in terms of $a,b,C$ and $N$).

Since there are only finitely many monomials with bounded multiplicity, we conclude that for any $N$-complement $\Gamma$, the value of $w(\Gamma)$ is determined by only finitely many such monomials. In other words, there are finitely many $\beta_{j}\in \bZ_{\ge 0}^q$ ($j\in J$), such that for any given $N$-complement $\Gamma=\frac{1}{N}\{g=0\}$ and any $w\in U$, we have
\[
w(\Gamma)=\frac{1}{N} \min_{j\in J_0}\{ \langle \alpha_w, \beta_j \rangle\},
\]
where $\alpha_w\in \QM_\eta(Y,E)\cong \bR^q_{\ge 0}$ is the coordinate of $w$ at $\eta$, and $J_0\subseteq J$ is the subset consisting of those indices $j$ for which the coefficients $c_{\beta_j}$ in the local expansion $g=\sum_{\beta} c_{\beta}z^{\beta}\in \hat{\cO}_{Y,\eta}$ are nonzero. So for an $N$-complement $\Gamma$, we have $v(\Gamma)\neq A_{X,\Delta+\varepsilon D}(v)$ if and only if
\[
A_{X,\Delta+\varepsilon D}(v)> \min_{j\in J_0} \langle \alpha_v, \beta_j \rangle
\]
for the subset $J_0$ corresponding to $\Gamma$ (note that $J_0$ is independent of $v$). For any such $J_0$, there is a neighborhood $U_{J_0}$ of $v$ such that for any $w\in U_{J_0}$, 
\[
A_{X,\Delta+\varepsilon D}(w)> \min_{j\in J_0} \langle \alpha_w, \beta_j \rangle \, .
\]
As there are only finitely many subsets $J_0$ of $J$, we conclude that for the pair $(X,\Delta+\varepsilon D)$, there exists a neighbourhood $U'\subseteq U$ of $v$, such that if $\Gamma$ is an $N$-complement of $(X,\Delta+\varepsilon D)$ for $N$ given as above, the following holds:
\[
v(\Gamma)\neq A_{X,\Delta+\varepsilon D}(v) \mbox{\ \ implies \ \ \ } w(\Gamma)\neq A_{X,\Delta+\varepsilon D}(w)\mbox{\  for any  }w\in U'\,.
\]
In particular, if we choose $w_0\in U'$, since $w_0(\Gamma_0)= A_{X,\Delta+\varepsilon D}(w_0)$ for an $N$-complement $\Gamma_0$ we constructed above, we have $v(\Gamma_0)= A_{X,\Delta+\varepsilon D}(v)$ and therefore $v$ is also an lc place of $(X,\Delta+\Gamma')$, where $\Gamma'=\varepsilon D+\Gamma_0$. Since $\pi_*^{-1}\Gamma'\ge \varepsilon G$ and $G$ is $\pi$-ample, we see that $\Gamma'$ is a special $\bQ$-complement with respect to $(Y,E)$. In other words, $v$ is a monomial lc place of a special $\bQ$-complement as desired.
\end{proof}

\subsection{Degeneration induced by a Koll\'ar model}

Next we recall the multiple divisor version of the Rees construction (see Section \ref{ss-kollar}) that was first considered in \cite{Xu-HRFG}. Let $x\in (X=\Spec(R),\Delta)$ be a klt singularity as before, and let $\pi\colon (Y,E)\to (X,\Delta)$ be a Koll\'ar model. 
Let $E_1,\dots,E_r$ be the irreducible components of $E$. Then there exists a $\bQ$-complement $\Gamma$ of $x\in (X,\Delta)$ such that $E_1,\dots,E_r$ are all lc places of $(X,\Delta+\Gamma)$. By Lemma \ref{lem:perturb qdlt boundary}, there exists some effective $\bQ$-divisor $\Delta_0$ on $Y$ such that $(Y,\Delta_0)$ is klt and $-(K_Y+\Delta_0)$ is ample over $X$. This condition is also satisfied by any small $\bQ$-factorial modification of $Y$, which exists by \cite{BCHM}*{Corollary 1.4.3}. By \cite{BCHM}*{Corollary 1.3.1}, this implies that the small $\bQ$-factorial modification of $Y$ is a Mori dream space, hence the $\bZ^k$-graded algebra
\begin{equation}
\cR_k = \bigoplus_{(i_1,\dots,i_k)\in \bZ^k} \pi_*\cO_Y(-i_1E_1-\dots -i_kE_k)~ t_1^{-i_1}\cdots t_k^{-i_k}    
\end{equation}
is finitely generated for any $k\le r$ and induces a $\bG_{m}^k$-equivariant family 
\begin{equation}\label{e-Rees}
f_k\colon \cX_k:=\Spec(\cR_k)\to \bA^k,
\end{equation}
where $t_1,\dots,t_k$ give the coordinates of $\bA^k$. It comes with a natural isomorphism
\[
\cX_k\times_{\bA^k} (\bA^1\setminus\{0\})^k \cong X\times (\bA^1\setminus\{0\})^k.
\]
Using this isomorphism, we let $\Delta_{\cX_k}$ (resp. $\Gamma_{\cX_k}$) be the closure of $\Delta\times (\bA^1\setminus\{0\})^k$ (resp. $\Gamma\times (\bA^1\setminus\{0\})^k$) in $\cX_k$. For simplicity, we will drop the subscript $k$ at various places when there is no confusion, e.g. when we work over a fixed $\bA^k$. Also let $H_i=(t_i=0)\subseteq \bA^k$ ($i=1,\dots,k$) be the coordinate hyperplanes.

The following statement is proved in \cite[Theorem 3.5 and Proposition 3.6]{Xu-HRFG}.

\begin{thm} \label{thm:locally stable family}
In the above notation, the family $f\colon (\cX,\Delta_{\cX}+\Gamma_{\cX})\to \bA^k$ is locally stable.
\end{thm}


To make our exposition more self-contained, we will sketch a proof of this theorem in the situation we consider above. In fact, one of our key observations for the proof of Theorem \ref{t-localHRFG} is to upgrade the above construction to obtain a locally stable family of Koll\'ar models $(\cY,\cE)\to (\cX,\Delta_{\cX})\to \bA^k$. This includes the enhancement of Theorem \ref{thm:locally stable family} that the locally stable family $(\cX,\Delta_{\cX})\to \bA^k$ has irreducible and klt fibers. As such, we will construct $\cX$ and $\cY$ simultaneously. This is given by the following two propositions.

\begin{prop}\label{pro-localstablefamily}
For any $k\le r$, the divisors $E_1,\dots,E_k$ induce a locally stable family $f\colon (\cX,\Delta_{\cX}+\Gamma_{\cX})\to \bA^k$ such that $(\cX,\Delta_{\cX}+\Gamma_{\cX}+f^*(H_1+\dots+H_k))$ is lc and for any $t\in \bA^k$, $(\cX,\Delta_{\cX})\times_{\bA^k}\{t\}$ is klt.
\end{prop}

\begin{prop} \label{pro:degeneration remain qdlt Fano}
Notation as in Proposition \ref{pro-localstablefamily}. Then 
$$\pi\times \mathrm{id}\colon (Y,E)\times (\bA^1\setminus \{0\})^k \to X\times (\bA^1\setminus \{0\})^k$$ extends to a locally stable family of Koll\'ar models $p\colon (\cY,\cE)\to (\cX,\Delta_{\cX})$. 
Moreover, 
\[
(\cY,p_*^{-1}\Delta_{\cX}+\cE+p^*f^*(H_1+\dots+H_k))
\]
is qdlt. In particular, the central fiber $\cX_0:=f^{-1}(0)$ is irreducible.
\end{prop}

We use induction to prove the above two statements together. We denote by \ref{pro-localstablefamily}$_k$ and \ref{pro:degeneration remain qdlt Fano}$_k$ that the corresponding statement in the proposition holds for $k$. When $k=0$, both statements are clearly true. 

\begin{proof}[Proof of \ref{pro-localstablefamily}$_k$ implies  \ref{pro:degeneration remain qdlt Fano}$_k$]
Let $H=H_1+\dots+H_k$, $D=\pi_*^{-1}\Delta$, and let $\cE_i$ $(1\le i \le r)$ be the divisor over $\cX$ given by the (closure of) $E_i\times (\bA^1\setminus\{0\})^k$. Since we assume  \ref{pro-localstablefamily}$_k$, the pair $(\cX,\Delta_{\cX}+\Gamma_{\cX}+f^*H)$ is lc; by construction, it has $\cE_1,\dots,\cE_r$ as lc places. Note that the lc centers of $(\cX,\Delta_{\cX})$ (if there is any) are contained in $f^*H$, as over $(\bA^1\setminus\{0\})^k$ the pair becomes $(X,\Delta)\times (\bA^1\setminus\{0\})^k$. But as $(\cX,\Delta_{\cX}+\Gamma_{\cX}+f^*H)$ is lc, this implies $(\cX,\Delta_{\cX})$ is klt. Using \cite{BCHM}*{Corollary 1.4.3} we see that there exists a projective birational morphism $p\colon \cY\to \cX$ that extracts all $\cE_i$ ($1\le i\le r$) as exceptional divisors. Since
for $0<\varepsilon \ll 1$, the crepant pull back of $K_{\cX}+\Delta_{\cX}+(1-\varepsilon)\Gamma_{\cX}$ yields a klt pair on $\cY$, we see that $\cY$ is of Fano type over $\cX$.
Therefore, as $-(K_Y+D+E)$ is ample, we may replace $\cY$ by the ample model of $-(K_{\cY}+\cD+\cE)$ over $\cX$ (where $\cE:=\cE_1+\dots+\cE_r$ and $\cD$ is the birational transform of $\Delta_{\cX}$; it is also the closure of $D\times (\bA^1\setminus\{0\})^k$) so that $$(\cY,\cD+\cE)\times_{\bA^k} (\bA^1\setminus \{0\})^k\cong (Y,D+E)\times (\bA^1\setminus \{0\})^k$$ and $-(K_{\cY}+\cD+\cE)$ is $p$-ample. It is in fact ample since $\cX$ is affine. 

Let $g=f\circ q\colon \cY\to \bA^k$. To show that $(\cY,\cD+\cE+g^*H)$ is qdlt, we proceed in several steps. First we show that $(\cY,\cD+\cE)$ is qdlt. This is certainly true over $(\bA^1\setminus\{0\})^k$, as the pair simply becomes $(Y,D+E)\times (\bA^1\setminus \{0\})^k$. On the other hand, we have
\[
p^*(K_{\cX}+\Delta_{\cX}+\Gamma_{\cX}+f^*H)\ge K_{\cY}+\cD+\cE+g^*H
\]
by the above construction, thus $(\cY,\cD+\cE+g^*H)$ is lc. For any divisor $F$ over $\cY$ whose center is contained in $g^*H$, we then have $$A_{\cY,\cD+\cE}(F)>A_{\cY,\cD+\cE+g^*H}(F)\ge 0.$$ This implies that none of the lc centers of $(\cY,\cD+\cE)$ are contained in $g^*H$ and hence the pair is qdlt. 

Let $$Z:=\bigcap_{i=1}^r E_i \mbox{\ \ \ and \ \ \ }\cZ:=\bigcap_{i=1}^r \cE_i.$$ By Proposition \ref{prop:km DMR is simplex} we know that $Z$ is non-empty and irreducible. We note that $\cZ$ is also irreducible. In fact, as 
\[\cZ\times_{\bA^k}(\bA^1\setminus\{0\})^k \cong Z \times (\bA^1\setminus\{0\})^k \,,
\]
we see that if $\cZ$ is reducible, then one of its components $S$ lies inside $g^*H$. But $S$ is necessarily an lc center of the qdlt pair $(\cY,\cD+\cE)$, and by the above discussion $g^*H$ does not contain any lc center of $(\cY,\cD+\cE)$, a contradiction. Thus $\cZ$ is irreducible as well. 

We next show that $\cZ_0=\cZ\cap g^{-1}(0)$ is the minimal lc center of $(\cY,\cD+\cE+g^*H)$ (in particular, $\cZ_0$ is irreducible). Indeed, by Lemma \ref{lem:lc center of Fano pair} we know that the (unique) minimal lc center $W$ of $(\cY,\cD+\cE+g^*H)$ intersecting $g^{-1}(0)$ is contained in $\cZ_0$, as $\cE_i$ $(1\le i \le r)$ and components of $g^*H_j$ $(j=1,\dots,k)$ are all lc centers of this pair. By construction, $\cY$ carries a $\bT=\bG_m^k$-action (induced by the $\bT$-action on $\cX$), hence $W$ is $\bT$-invariant. 
Suppose that $\cZ_0\neq W$, then since $|-m(K_{\cY}+\cD+\cE)\otimes \cI_W|$ is globally generated for some sufficiently divisible integer $m>0$, we may find some irreducible $\bT$-invariant divisor $$\cG\in |-m(K_{\cY}+\cD+\cE)|$$ such that $W\subseteq \Supp(\cG)$ but $\cZ_0\not\subseteq \Supp(\cG)$. By $\bT$-invariance, we have $\cG$ is the closure of $G\times (\bA^1\setminus\{0\})^k$ for some divisor $$G\in |-m(K_Y+D+E)|.$$ By construction $Z\not\subseteq \Supp(G)$ (as $\cZ_0\not\subseteq \Supp(\cG)$) and therefore $G$ does not contain any lc center of $(Y,D+E)$. It follows that $(Y,D+\varepsilon G+E)$ is still qdlt and $-(K_Y+D+\varepsilon G+E)$ is ample when $0<\varepsilon\ll 1$. Thinking of $\pi_*(D+\varepsilon G)$ as playing the role of the original boundary divisor $\Delta$, the same argument from the previous paragraphs implies that the specialization $(\cY,\cD+\varepsilon \cG+\cE+g^*H)$ is lc. But this is impossible as $\cG$ contains the minimal lc center of $(\cY,\cD+\cE+g^*H)$. The obtained contradiction implies that $\cZ_0$ is the minimal lc center of $(\cY,\cD+\cE+g^*H)$.

Next we aim to show that each $\cE_i$ ($1\le i\le r$) is $\bQ$-Cartier at the generic point of $\cZ_0$. Since $\cL:=-(K_{\cY}+\cD+\cE)$ is $\bQ$-Cartier, this is true if we can find a divisor in some $|m\cL-\ell \cE_i|$ ($1\le i\le r$) whose support does not contain $\cZ_0$. To this end, let $L:=-(K_Y+D+E)$, let $\ell$ be a positive integer such that $\ell E_i$ is Cartier at the generic point of $Z$ (this is possible since $Y$ has quotient singularity at the generic point of $Z$), let $m>0$ be a sufficiently divisible integer and let $B_-$ (resp. $B_+$) be a general member of $|mL-\ell E_i|$ (resp. $|mL+\ell E_i|$). By Lemma \ref{lem-noncartierlocus}, neither $B_-$ nor $B_+$ contains $Z$ in its support, thus none of the lc centers of $(Y,D+E)$ are contained in $\Supp(B_- + B_+)$. It follows that the pair $$(Y,D+\varepsilon(B_-+B_+)+E)$$ remains qdlt when $0<\varepsilon\ll 1$ (note that $B_-+B_+\sim 2mL$ is Cartier). 
Let $\cB_{-}$ (resp. $\cB_{+}$) be the closure of $B_-\times (\bA^1\setminus\{0\})^k$ (resp. $B_+\times (\bA^1\setminus\{0\})^k$). As before, this implies that the corresponding pair $$(\cY,\cD+\varepsilon(\cB_{-}+\cB_{+})+\cE+g^*H)$$ over $\bA^k$ is lc. In particular, $\Supp(\cB_{-}+\cB_{+})$ does not contain $\cZ_0$ which is an lc center of $(\cY,\cD+\cE+g^*H)$. 
From the construction, the difference between $\cB_-$ and $m\cL-\ell \cE_i$ yields a divisor supported on $g^*H$. However, $g^*H_j$ ($1\le j \le k$) is irreducible as $f^*H_j$ is irreducible, and ${\rm Ex}(p)$ supports on $\cE$ which does not contain any component supported over $g^*H$. Thus the difference is linearly equivalent to zero and $\cB_-$
gives the sought divisor in $|m\cL-\ell \cE_i|$. By the discussion before, we deduce that $\cE_i$ is $\bQ$-Cartier at the generic point of $\cZ_0$.

We can now finish the proof: it is already shown that $\cZ_0$ is the unique minimal lc center of $(\cY,\cD+\cE+g^*H)$, and by Lemma \ref{lem:lc center of Fano pair} every lc center of $(\cY,\cD+\cE+g^*H)$ contains $\cZ_0$, which has codimension $r+k$ in $\cY$. Thus it remains to show that $(\cY,\cD+\cE+g^*H)$ is the quotient of an SNC pair at the generic point $\eta$ of $\cZ_0$. We already know that every $\cE_i$ $(1\le i\le r)$ is $\bQ$-Cartier at $\eta$, while each $g^*H_j$ ($1\le j\le k$) is clearly Cartier. Hence the statement follows from \cite{K+92}*{18.22 Theorem and 18.23 Complement}, where we use $g^*H_j$ is irreducible for every $j$ as argued above. 
 As $(\cY,\cD+\cE+g^*H)$ is qdlt, this implies that $$g^{-1}(0)=g^*H_1\cap \dots\cap g^*H_k$$ is normal by Lemma \ref{lem:qdlt->normal lc center}. Since it is also connected, we conclude that $g^{-1}(0)$ is irreducible. It follows that $\cX_0$ is also irreducible and we finish the proof.
\end{proof}

\begin{proof}[Proof of \ref{pro-localstablefamily}$_{k}+$\ref{pro:degeneration remain qdlt Fano}$_k$ implies \ref{pro-localstablefamily}$_{k+1}$]
Proposition \ref{pro:degeneration remain qdlt Fano}$_k$ gives us a locally stable family of Koll\'ar models $(\cY,\cE=\sum^r_{i=1}\cE_i)\stackrel{p}{\to} (\cX={\rm Spec}(\cR_k),\Delta_{\cX})\to \bA^k$. Applying Lemma \ref{lem:flat limit}(4) to this family implies $\cE_{k+1}$ can be extracted over $(\cX,\Delta_{\cX})$ as a locally stable family of Koll\'ar components, such that for any point $t\in \bA^k$, $\cE_{k+1}\times_{\bA^k}\{t\}$ is an lc place of $(X_t,\Delta_t+\Gamma_t)$, where the latter is the fiber of $(\cX,\Delta_{\cX}+\Gamma_{\cX})$ over $t$.

Then we can consider 
$$\cR':=\bigoplus_{i\in\bZ} p_* \mathcal{O}_{\cY}(-i\cE_{k+1})t_{k+1}^{-i},$$
which is a $\bk[t_1,\dots,t_{k+1}]$-algebra, and there is a morphism 
$$f'\colon  \cX':={\rm Spec}(\cR')\to \bA^{k+1}.$$
Let $\Delta_{\cX'}$ and $\Gamma_{\cX'}$ be the closures of $\Delta_{\cX}\times \{\bA^1\setminus 0\}$ and $\Gamma_{\cX}\times \{\bA^1\setminus 0\}$.
Since $\cE_{k+1}$ is a locally stable family of Koll\'ar components over $(\cX,\Delta_{\cX})$ and an lc place of $(\cX,\Delta_{\cX}+\Gamma_{\cX})$, $(\cX',\Delta_{\cX'})$ is a locally stable family over $\bA^{k+1}$ (see \cite[Lemma 3.3]{Xu-HRFG}). Since $\cE_{k+1}$ is the closure of $E_{k+1}\times (\bA^1\setminus\{0\})^k$, we get
\[
p_*\cO_{\cY}(-i\cE_{k+1}) = \cR_k \cap \left(\bigoplus_{(i_1,\dots,i_k)\in\bZ^k} \pi_*\cO_Y(-iE_{k+1})t_1^{-i_1}\cdots t_k^{-i_k}\right),
\]
using \ref{pro-localstablefamily}$_{k}$ we deduce that $\cR'\cong \cR_{k+1}$ which means $\cX'$ coincides with $\cX_{k+1}$.

Moreover, by Lemma \ref{lem:flat limit}(4), the above Rees construction commutes with base change. It follows from Lemma \ref{lem-kollardegenerate} that, for any $t\in \bA^k$ with the Koll\'ar component $\cE_{k+1}\times_{\bA^k}\{t\}$ over $(\cX_t,\Delta_{\cX_t}):=(\cX,\Delta_{\cX})\times_{\bA^k}\{t\}$ which is a lc place of a $\bQ$-complement $\Gamma_t$, the induced Rees construction $f'_t\colon \cX'_t\to \bA^1$ satisfies $(\cX'_t,\Delta_{\cX'_t}+\Gamma_{\cX'_t}+{f'_t}^{-1}(0))$ is log canonical (where $\Delta_{\cX'_t}$ and $\Gamma_{\cX'_t}$ are respectively the closure of $\Delta_{\cX_t}\times (\bA^1\setminus\{0\})$ and $\Gamma_{\cX_t}\times (\bA^1\setminus\{0\})$) and $(\cX'_t,\Delta_{\cX'_t}+{f'_t}^{-1}(0))$ is plt. 
By adjunction, the central fiber of $f'_t$ is klt. By inversion of adjunction, it also implies that for $H'_{j}=(t_j=0)\subset \bA^{k+1}$ 
 $(1\le j\le k+1)$, 
 $$(\cX',\Delta_{\cX'}+\Gamma_{\cX'}+{f'}^*(H'_1+\dots+H'_{k+1}))$$ is lc (first in a neighbourhood of $\cX_t$, then extend to the entire $\cX'$ by the $\bG_m^{k+1}$-action). In particular, the fibers of $(\cX',\Delta_{\cX'})\to \bA^{k+1}$ are klt and $(\cX',\Delta_{\cX'}+\Gamma_{\cX'})$ is also a locally stable family over $\bA^{k+1}$. This proves \ref{pro-localstablefamily}$_{k+1}$.
\end{proof}

\subsection{Irreducible central fiber implies finite generation}\label{ss-irredu=finite}

Finally, we will show that the irreducibility of the central fiber, as proven by Proposition \ref{pro:degeneration remain qdlt Fano} implies the finite generation. This was pointed out in \cite{Xu-HRFG}. Here we give a detailed argument in a slightly more general context. 

We fix the following setup. Let $x\in X=\Spec(R)$ be a singularity and let $$\pi\colon (Y,E=E_1+\dots+E_m)\to X$$ be a model. Assume that:
\begin{itemize}
    \item $\bigcap_{i=1}^m E_i$ is nonempty,
    \item $(Y,\Supp(E+\Ex(\pi)))$ is toroidal at a generic point $\eta$ of $\bigcap_{i=1}^m E_i$, and
    \item there exists an effective $\bQ$-Cartier $\bQ$-divisor $F$ on $Y$ such that $-F$ is $\pi$-ample and $E+\Ex(\pi)+F$ has toroidal support at $\eta\in Y$.
\end{itemize}

\begin{rem}
The first two conditions allow us to define the simplex $\QM_\eta(Y,E)\subseteq \Val_X$ and normally it contains the valuation we want to analyze. The last condition rules out some peculiar birational maps but in the situations we consider it is often automatic. For example, if $(Y,E)$ is a Koll\'ar model of a klt singularity $x\in (X,\Delta)$, then the first condition follows from Proposition \ref{prop:km DMR is simplex}, the second condition holds by definition, and for the last condition we can take
\[
F=K_Y+\pi_*^{-1}\Delta+E-\pi^*(K_X+\Delta),
\]
which is supported on $E$. In general, for any quasi-monomial valuation $v\in \Val_X$, we can always choose a log smooth model $\pi\colon (Y,E)\to X$ such that $\Ex(\pi)$ supports a $\pi$-ample divisor (which serves as our $-F$; by the negativity lemma $F$ is effective) and $v\in \QM(Y,E)$.
\end{rem}

Let $v_i\in\Val_{X,x}$ ($i=1,\dots,r$) be divisorial valuations that all lie in $\QM_\eta(Y,E)$. There exists a natural linear map $\bR^r_{\ge 0}\to \QM_\eta(Y,E)$ sending the $i$-th basis vector $e_i$ to $v_i$. For $0\neq s=(s_1,\dots,s_r)\in \bR^r_{\ge 0}$ we let $v_s\in \QM_\eta(Y,E)$ be the image of $s$. For any $\lambda\in \bR$, we also let $I_{s,\lambda}\subseteq R$ be the ideal given by 
\begin{equation}\label{e-defineI}
I_{s,\lambda} = {\rm Span}\{f\in R\,|\,s_1 v_1(f)+\dots+s_r v_r(f)\ge \lambda\}.
    \end{equation}

The above expression is chosen so that $I_{s,\lambda}$ ``looks like'' the valuation ideal of $v_s$. In fact, this is the case if $Y=X$ (by the definition of quasi-monomial valuations). The following lemma is an algebraic reformulation (in our more general setup) of the result claimed in \cite{Xu-HRFG}*{Lemma 3.8}.

\begin{lem} \label{lem:induced valuation=v_t}
Let $s\in \bR^r_{\ge 0}\setminus\{0\}$. Assume that there exists some valuation $w\in \Val_{X,x}$ such that $I_{s,\lambda}=\fa_\lambda(w)$ for all $\lambda$. Then $w=v_s$. 
\end{lem}

Geometrically, the condition that $I_{s,\bullet}$ are valuation ideals is equivalent to saying that the central fiber of the induced degeneration is irreducible and reduced, i.e. the graded algebra $\bigoplus_{\lambda} (I_{s,\lambda}/I_{s,>\lambda})$ is integral (see Lemma \ref{lem:graded integral->valuation}).

In the global setting, the analogous construction for $I_{s,\lambda}$ is often viewed as the ``geodesic'' connecting two valuations (see \cite{Reb20}, \cite{BLXZ}*{Section 3.1.2} or \cite{BLQ}*{Section 4}). Thus Lemma \ref{lem:induced valuation=v_t} can be interpreted as saying that whenever the geodesic lies in the valuations space, it is the obvious line in the corresponding simplex $\QM(Y,E)$.

\begin{proof}
It is clear from the definition of quasi-monomial valuations that $I_{s,\lambda}\subseteq \fa_\lambda(v_s)$ for all $\lambda$, hence $v_s\ge w$ on $R$. It remains to show that $w\ge v_s$.  Let $0\neq f_0\in R$ and let $\lambda_0=v_s(f_0)$. Let $$\fb=\big\{f\in \cO_Y\mid v_s(f)\ge \lambda_0\big\}$$ be the corresponding valuation ideal on $Y$; in particular $f_0\in \fb$. Let $\alpha_i=v_s(E_i)$. Then we have a surjection
\begin{equation} \label{eq:monomial generation}
    \bigoplus_{b_i\in\bN,\sum\alpha_i b_i\ge \lambda_0} \cO_Y\left(-\sum_{i=1}^m b_i E_i\right)\to \fb
\end{equation}
by the definition of the quasi-monomial valuation $v_s$. Let $\ell>0$ be a sufficiently divisible integer; in particular, 
$$\pi^*\pi_*\cO_Y(-\ell F)\to \cO_Y(-\ell F)$$
is surjective. We claim that for a general element $f\in \pi_*\cO_Y(-\ell F)$, we have
\[
w(f)=v_s(f)\mbox{\ \ \  and \ \ \ }w(f_0 f)\ge v_s(f_0 f),
\]
which together implies $w(f_0)\ge v_s(f_0)$. Since $f_0$ is arbitrary, we deduce that $w\ge v_s$ as desired.

We now prove the claim. First observe that since $f\in \pi_*\cO_Y(-\ell F)$ is general, we have $$\pi^*({\rm div}(f))=\ell F+D$$ for some divisor $D$ that does not contain any stratum of $E$. Thus 
$$s_1 v_1(f)+\dots+s_r v_r(f)=s_1 v_1(\ell F)+\dots+s_r v_r(\ell F)=v_s(\ell F)=v_s(f),$$ where the second equality follows from the definition of the valuation $v_s$ and the fact that the local equation of $\ell F$ is given by a monomial. In particular, 
$$f\in I_{s,v_s(f)}=\fa_{v_s(f)}(w),$$
i.e. $w(f)\ge v_s(f)$ which together with $v_s\ge w$ implies $w(f)=v_s(f)$.

Next we show that $w(f_0 f)\ge v_s(f_0 f)$. Since $\ell$ is sufficiently divisible and $-F$ is $\pi$-ample, the map 
\[
\bigoplus_{b_i\in\bN,\sum\alpha_i b_i\ge \lambda_0} \pi_*\cO_Y\left(-\sum_{i=1}^m b_i E_i-\ell F\right)\to \pi_*(\fb\otimes\cO_Y(-\ell F))
\]
induced by \eqref{eq:monomial generation} is surjective. As $$f_0 f\in \pi_*(\fb\otimes\cO_Y(-\ell F)),$$ we may write $$f_0 f=g_1+\dots+g_p,$$ where each $$g_j\in \pi_*\cO_Y\left(-\sum_{i=1}^m b^{(j)}_i E_i-\ell F\right)$$ for some $b^{(j)}_i\in\bN$ that satisfies $\sum\alpha_i b^{(j)}_i\ge \lambda_0$. Similar to the argument above, for any $j=1,\dots,p$ we have
\begin{align*}
    \sum_{i=1}^r s_i v_i(g_j) & \ge \sum_{i=1}^r s_i v_i\left(\sum_{i=1}^m b^{(j)}_i E_i+\ell F\right) \\
    & =v_s\left(\sum_{i=1}^m b^{(j)}_i E_i+\ell F\right) = \sum_{i=1}^m \alpha_i b^{(j)}_i+v_s(\ell F)\\
    & \ge \lambda_0 + v_s(\ell F) = v_s(f_0)+v_s(f) = v_s(f_0 f),
\end{align*}
where the first equality follows from the definition of the quasi-monomial valuation $v_s$ as above. It follows that each $g_j\in I_{s,v_s(f_0 f)}$ which gives $f_0 f\in I_{s,v_s(f_0 f)}$. Therefore, $w(f_0 f)\ge v_s(f_0 f)$. This proves the claim and completes the proof of the lemma.
\end{proof}

In practice, it is often hard to show directly that $I_{s,\bullet}$ are valuations ideals when the weights $s_i$ are irrational. The following auxiliary lemma allows us to only consider rational weights.

\begin{lem} \label{lem:a=I irrational weights}
Assume that $\fa_\lambda (v_s)=I_{s,\lambda}$
for all $\lambda\in\bR$ and all $s\in \bQ^r_{\ge 0}\setminus 0$. Then the same holds when $s\in \bR^r_{\ge 0}\setminus 0$.
\end{lem}

\begin{proof}
We may write $s=\sum_{i=1}^m c_i s^{(i)}$ for some $s^{(i)}\in\bN^r\setminus \{0\}$ ($i=1,\dots,m$) and some $c_1,\dots,c_m\in \bR_+$ that are $\bQ$-linearly independent. Note that it suffices to prove 
$$\fa_\lambda(v_s)=I_{s,\lambda} \mbox{ when $\lambda=\sum_{i=1}^m a_i c_i$ for some }a_1,\dots,a_m\in \bN,$$ 
since these are the only places where $\fa_\lambda(v_s)$ or $I_{s,\lambda}$ possibly jumps. 

Fix any such $s$ and $\lambda$. We claim that for any $c'_1,\dots,c'_m\in\bQ_+$ that are sufficiently close to $c_1,\dots,c_m$ respectively, we have
\[
\fa_\lambda(v_s)=\fa_{\lambda'}(v_{s'})\quad \mathrm{and}\quad I_{s,\lambda}=I_{s',\lambda'},
\]
where $s'=\sum_{i=1}^m c'_i s^{(i)}$ and $\lambda'=\sum_{i=1}^m a_i c'_i$. Since $s'\in \bQ^r_{\ge 0}$, we have $\fa_{\lambda'}(v_{s'})=I_{s',\lambda'}$ by assumption, hence the lemma follows immediately from this claim. 

We now proceed to prove the above claim. Let us first check that $$\fa_\lambda(v_s)=\fa_{\lambda'}(v_{s'}).$$ Suppose that this is not the case, i.e. there exists some $0\neq f\in R$ that is contained in one of $\fa_\lambda(v_s),\fa_{\lambda'}(v_{s'})$ but not the other. Since each $c'_i$ is sufficiently close to $c_i$, we may assume that $s\le 2s'$ and $\lambda'\le 2\lambda$. By the definition of the valuations $v_s$, we then have $v_s\le 2v_{s'}$. It follows that $v_s(f)<4\lambda$, otherwise $$v_{s'}(f)\ge \frac{1}{2}v_s(f)\ge 2\lambda\ge \lambda'$$ and thus $$f\in \fa_\lambda(v_s)\cap\fa_{\lambda'}(v_{s'}),$$ a contradiction. Let $\eta\in Y$ be the center of $v_s$. Recall that $v_s(f)$ is the smallest weight of monomials in the local expansion of $f\in \cO_{Y,\eta}$. Consider the monomials with weights at most $4\lambda$. Each monomial has weight $\sum_{i=1}^m b_i c_i$ for some $b_i\in \bN$. If we perturb $c_i$ to $c'_i$, then the weight becomes $\sum_{i=1}^m b_i c'_i$. As $c_1,\dots,c_m$ are $\bQ$-linearly independent, we have 
\[
\sum_{i=1}^m b_i^{(1)} c_i = \sum_{i=1}^m b_i^{(2)} c_i \mbox{\ \  if and only if \ \ }
(b_1^{(1)},...,b_m^{(1)})=(b_1^{(2)},...,b_m^{(2)}).
\]
Combined with the fact that there are only finitely many monomials with weights at most $4\lambda$, we see that if the $c'_i$'s are sufficiently close to $c_i$, then 
\begin{equation} \label{eq:wt ineq preserved}
    4\lambda>\sum_{i=1}^m b_i^{(1)} c_i\ge \sum_{i=1}^m b_i^{(2)} c_i\mbox{\ \  if and only if \ \ }4\lambda>\sum_{i=1}^m b_i^{(1)} c'_i\ge \sum_{i=1}^m b_i^{(2)} c'_i;
\end{equation}
for any $(b_1^{(1)},...,b_m^{(1)})$ and $(b_1^{(2)},...,b_m^{(2)})$ in $\bN^m$ such that $\sum_{i=1}^m b_i^{(j)} c_i<4\lambda$ $(j=1,2)$.
In other words, the monomials that are responsible for computing $v_s(f)$ are also responsible for computing $v_{s'}(f)$. In particular, we have 
$$v_s(f)\ge \sum_{i=1}^m a_i c_i=\lambda\mbox{\ \  if and only if \ \ }v_{s'}(f)\ge \sum_{i=1}^m a_i c'_i=\lambda',$$ hence it cannot happen that $f$ belongs to one of $\fa_\lambda(v_s),\fa_{\lambda'}(v_{s'})$ but not the other. This is a contradiction. Thus $\fa_\lambda(v_s)=\fa_{\lambda'}(v_{s'})$ as desired.

The other equality $I_{s,\lambda}=I_{s',\lambda'}$ can be proved in a similar fashion. Suppose that $I_{s,\lambda}\neq I_{s',\lambda'}$, then there exists some $f\in R$ such that one of the inequalities $$\sum_{i=1}^r s_i v_i(f)\ge \lambda,\ \ \sum_{i=1}^r s'_i v_i(f)\ge \lambda'$$ holds but not the other. As in the above proof, we may assume that $\sum_{i=1}^r s_i v_i(f)<4\lambda$. Note that 
$$\sum_{i=1}^r s_i v_i(f)=\sum_{i=1}^m b_i c_i\mbox{\ \  and \ \ }\sum_{i=1}^r s'_i v_i(f)=\sum_{i=1}^m b_i c'_i$$ for some $b_i\in\bN$. As in \eqref{eq:wt ineq preserved}, this implies that 
\[
\sum_{i=1}^r s_i v_i(f)\ge \lambda\mbox{\ \  if and only if \ \ }\sum_{i=1}^r s'_i v_i(f)\ge \lambda',
\]
a contradiction. Thus $I_{s,\lambda}=I_{s',\lambda'}$ as desired. This proves the claim and also finishes the proof of the lemma. 
\end{proof}

Combining the above results with the analysis in the previous subsection, we can now prove the direction $(3)\Rightarrow (1)$ of Theorem \ref{t-localHRFG}. 

\begin{cor} \label{cor:km imply fg}
Let $x\in (X=\Spec(R),\Delta)$ be a klt singularity and $\pi\colon (Y,E)\to (X,\Delta)$ a Koll\'ar model. Assume that $v_i=\ord_{E_i}$ $(1\le i\le r=m)$ are given by the components of $E$. Then $\fa_\lambda (v_s)=I_{s,\lambda}$ for all $\lambda\in\bR$ and all $s\in \bR^r_{\ge 0}\setminus 0$. In particular, the graded algebra $\gr_{v_s} R$ is finitely generated, and $\gr_{v_s} R\cong \gr_{v_{s'}} R$ whenever $s$ and $s'$ lies on the interior of the same face of $\bR^r_{\ge 0}\setminus 0$.
\end{cor}

\begin{proof}
By Propositions \ref{pro-localstablefamily} and \ref{pro:degeneration remain qdlt Fano}, we have a $\bG_m^r$-equivariant family 
\[
\cX=\Spec\left(\bigoplus_{(i_1,\dots,i_r)\in \bZ^r} \pi_*\cO_Y(-i_1E_1-\dots -i_rE_r)~ t_1^{-i_1}\cdots t_r^{-i_r}\right)\to \bA_{t_1,\dots,t_r}^r
\]
with integral central fiber. By Lemma \ref{lem:a=I irrational weights}, to show that $\fa_\lambda (v_s)=I_{s,\lambda}$,  it suffices to verify it when $s\in \bQ^r_{\ge 0}\setminus 0$. By rescaling, we may further assume that $s\in \bN^r$. The base change $\cX\times_{\bA^r} \bA^1$ via the map $\bA^1\to\bA^r: t\mapsto (t^{s_1},\dots,t^{s_r})$ is a test configuration given by
\[
\cX\times_{\bA^r} \bA^1=\Spec\left( \sum_{(i_1,\dots,i_r)\in\bZ^r} \pi_*\cO_Y(-i_1E_1-\dots -i_rE_r)~t^{-i_1 s_1-\dots-i_r s_r}\right)=\Spec\bigoplus_{i\in\bZ}t^{-i}I_{s,i}.
\]
Since the central fiber $\Spec(\bigoplus_{i\in\bN} I_{s,i}/I_{s,i+1})$ is integral and all $I_{s,i}$ are $\fm_x$-primary when $i>0$, by Lemma \ref{lem:graded integral->valuation} there exists some valuation $w\in \Val_{X,x}$ such that $I_{s,\lambda}=\fa_\lambda(w)$ for all $\lambda\in \bR$. By Lemma \ref{lem:induced valuation=v_t}, we have $w=v_s$ and hence $\fa_\lambda(v_s)=I_{s,\lambda}$ as desired.

We next show that $\gr_{v_s} R$ is finitely generated. Since $Y$ is of Fano type over $X$, the graded algebra
\[
\cR:=\bigoplus_{(i_1,\dots,i_r)\in\bN^r} \pi_*\cO_Y(-i_1 E_1-\dots-i_r E_r).
\]
is finitely generated by \cite{BCHM}*{Corollary 1.3.2}. On the other hand, since $\fa_\lambda(v_s)=I_{s,\lambda}$ for any $s\in \bR^r_{\ge 0}\setminus 0$, the natural map $\cR\to \gr_{v_s} R$ sending $$f\in \pi_*\cO_Y(-i_1 E_1-\dots-i_r E_r)\mapsto \bar{f}\in \gr_{v_s}^{i_1 s_1+\dots+i_r s_r} R$$ is surjective, by \eqref{e-defineI}. Therefore, $\gr_{v_s} R$ is also finitely generated.

Finally, assume that $s,s'\in \bR^r_{\ge 0}\setminus 0$ belong to the interior of the same face. If we also have $s,s'\in \bQ^r$, then the central fibers of the two induced test configurations $\cX\times_{\bA^r}\bA^1$ are isomorphic, since they are both given by the same fiber of $\cX\to \bA^r$. Since the central fibers are also given by $\Spec(\gr_{v_s} R)$ and $\Spec(\gr_{v_{s'}} R)$ respectively, this implies that $\gr_{v_s} R\cong \gr_{v_{s'}} R$. In general $s,s'$ may have irrational weights, but by \cite{LX18}*{Lemma 2.10} we have $\gr_{v_s} R \cong \gr_{v_u} R$ for some $u\in \bQ_{\ge 0}^r$ that lies on the interior of the same face that contains $s$. Similarly for $s'$. Thus the isomorphism $\gr_{v_s} R\cong \gr_{v_{s'}} R$ follows from the rational case treated above.
\end{proof}

We are now ready to prove the main result (Theorem \ref{t-localHRFG}) of this section, as well as its implications mentioned in the introduction.

\begin{proof}[Proof of Theorem \ref{t-localHRFG}]
By Lemma \ref{lem:fg->special complement} we have $(1)\Rightarrow (2)$. By Corollary \ref{cor:km for monomial lc place} we obtain $(2)\Rightarrow (3)$. Finally Proposition \ref{pro-localstablefamily} and Corollary \ref{cor:km imply fg} imply $(3)\Rightarrow (1)$.
\end{proof}

\begin{proof}[Proof of Corollary \ref{cor-global}]Let $L:=-r(K_X+\Delta)$, let $L^*$ be the dual line bundle and let $$(C,\Delta_C):=C(X,\Delta;L)$$ be the cone over $(X,\Delta)$ with the vertex $o$. In particular, $C={\rm Spec}\bigoplus_{m\in \bN}H^0(mL)$. Let $(Y_L,E_L)$ be the total space of $\pi^*(L^*)$. There is a morphism $(Y_L,E_L)\to (C,\Delta_C)$ which factors through the total space of $L^*$ over $(X,\Delta)$. Then $\Gamma_C$ (the cone over $\Gamma$) gives a $\bQ$-complement of $o\in (C,\Delta_C)$, and the Gauss extension $w\in \Val_{C,o}$ of $v$, defined as $w(s)=v(s)+mr$ for all $s\in H^0(mL)$, is an lc place of $(C,\Delta_C+\Gamma_C)$. Moreover, $\Gamma_C$ is special with respect to $(Y_L,E_L)$, as $\Gamma$ is special with respect to $(Y,E)$.

Since after a grading shift, we have
\[
{\rm gr}_w\Big(\bigoplus_{m\in \bN}H^0(-mr(K_X+\Delta))\Big)\cong {\rm gr}_v\Big(\bigoplus_{m\in \bN}H^0(-mr(K_X+\Delta))\Big),
\]
hence the latter is finitely generated. 
\end{proof}

\begin{proof}[Proof of Theorem \ref{thm-localfinite}]
By Lemma \ref{lem-localmonomial} the minimizer $v$ of $\hvol_{X,\Delta}$ is a monomial lc place of some special $\bQ$-complement $\Gamma$ with respect to a log smooth model $(Y,E)\to (X,\Delta)$. Then we can apply Theorem \ref{t-localHRFG}. 
\end{proof}

\begin{proof}[Proof of Theorem \ref{thm-SDC}] As we explained in the introduction, the only missing part was Theorem \ref{thm-localfinite}, which is established now. 
\end{proof}

\begin{proof}[Proof of Theorem \ref{t-algebraicvolume}]Let $v\in \Val_{X,x}$ be a minimizer of $\hvol_{X,\Delta}$. Denote by $X={\rm Spec}(R)$. By Theorem \ref{thm-localfinite}, ${\rm gr}_v(R)$ is finitely generated. If the rational rank of $v$ is $r$, then $X_v={\rm Spec}({\rm gr}_v(R))$ admits a $\bG_m^r$-action, and $(X_v,\Delta_v)$ is a klt singularity. Moreover, the valuation $v\colon R\setminus\{0\}\to \mathbb{R}$ yields a Reeb vector $\xi_v$ and $(X_v,\Delta_v;\xi_v)$ is a K-semistable log Fano cone singularities. By \cite{LX18}*{Lemma 2.58}, we have
\[
\hvol_{X,\Delta}(v)=\hvol_{X_v,\Delta_v}(\xi_v),
\]
thus we need to show that the right hand side is an algebraic number. 

This was proved in \cite{DS-GHlimit}, inspired by arguments in \cite{MSY08}. To see it, denote by $N={\rm Hom}(\bG_m,\bG_m^r)$ the coweight lattice, and $N_{\bR}=N\otimes_{\bZ}\bR$. Then by \cite{DS-GHlimit}*{Lemma 4.2}, the function 
$$\vol_{X_v,\Delta_v}\colon \xi \mapsto \vol({\rm wt}_{\xi})$$ (where ${\rm wt}_{\xi}$ is the valuation corresponding to $\xi$), defined on the Reeb cone $\mathfrak{t}_+\subseteq N_{\bR}$, can be extended to a rational function with rational coefficients on $N_{\bR}$. Moreover,  if we restrict on the hyperplane $A_{X_v,\Delta_v}({\rm wt}_{\xi})=1$, then the minimizer $\xi_v$ of $\vol$ has algebraic numbers as coordinates in $N_{\bR}$ by \cite[Lemma 4.6]{DS-GHlimit}.  This implies $\vol_{X_v,\Delta_v}(\xi_v)$ is an algebraic number. 
\end{proof}
\begin{proof}[Proof of Theorem \ref{thm-finitemodel}]
By the definition of quasi-monomial valuations, $v\mapsto A_{X,\Delta}(v)$ is linear on $\QM(Y,E)$ and $v\mapsto v(\Gamma)$ is piecewise linear with rational coefficients. The latter is also concave by Lemma \ref{lem-valuationconvexity}. Since $A_{X,\Delta}(v)\ge v(\Gamma)$, we see that 
\[
\LC(\Gamma;Y,E)=\{v\in \QM(Y,E)\mid A_{X,\Delta}(v)=v(\Gamma)\}
\]
is a rational polyhedral subset. We get $\LC_x(\Gamma;Y,E)$ by removing from $\LC(\Gamma;Y,E)$ its intersection with a few faces of $\QM(Y,E)$, thus it is also polyhedral. 

For the other statement, it suffices to consider the case when $v,w\in {\rm int}(\sigma)$ for some rational simplex $\sigma\subseteq \LC_x(\Gamma;Y,E)$, as the general case then follows by covering $\LC_x(\Gamma;Y,E)$ with rational simplices. By Theorem \ref{thm:construction of qdlt fano type model}, there exists a Koll\'ar model $\pi\colon (Y',E')\to (X,\Delta)$ such that $\QM(Y',E')=\sigma$. By Corollary \ref{cor:km imply fg}, we get $\gr_v R\cong \gr_w R$. The proof is now complete.
\end{proof}

\bibliography{ref}

\end{document}